\documentclass[12pt]{scrartcl}

\usepackage{setspace}

\usepackage{color}
\usepackage{array}
\usepackage{supertabular}
\usepackage{hhline}
\usepackage{multirow}
\usepackage{multicol}
\usepackage{textgreek}
\usepackage{upgreek}

\usepackage{wrapfig}
\usepackage[title]{appendix}
\usepackage{mathbbol}

\usepackage[listings,theorems]{tcolorbox}

\usepackage{pifont}

\usepackage[thicklines]{cancel}

\usepackage[LGR,T1]{fontenc}

\usepackage{tikz}

\usetikzlibrary{hobby}

\usepackage[customcolors]{hf-tikz}

\usepackage{circuitikz}

\usetikzlibrary{decorations.pathreplacing,arrows.meta}

\usetikzlibrary{calc,patterns,decorations.markings}
\usetikzlibrary{fadings}
\usetikzlibrary{angles,quotes}
\usepackage{tikz-3dplot}

\hfsetfillcolor{yellow!10}
\hfsetbordercolor{darkblue}

\def\uphrulefill{%
  \textcolor{darkblue!70}{\leavevmode\leaders\hrule height 1.5pt\hfill\kern0pt \rule[-0.56ex]{1.5pt}{5pt}}\\}
  
  \def\downhrulefill{%
  \textcolor{darkblue!70}{\leavevmode\leaders\hrule height 1.5pt\hfill\kern0pt \rule[0.0ex]{1.5pt}{5pt}}\\}

\definecolor{darkgreen}{rgb}{0.0, 0.2, 0.13}
\definecolor{darkred}{rgb}{0.55, 0.0, 0.0}
\definecolor{darkgreen}{rgb}{0.0, 0.2, 0.13}
\definecolor{darkmagenta}{rgb}{0.55, 0.0, 0.55}

\usepackage{pgfplots}
\usepackage{amsthm}
\pgfplotsset{compat=1.10}
\usepgfplotslibrary{fillbetween}
\usetikzlibrary{patterns}

 \usetikzlibrary{patterns}

\usepackage[latin1]{inputenc}

\newtheorem{theorem}{Theorem}[section]

\newtheorem{coro}{Corollary}[section]
\newtheorem{definition}{Definition}[section]
\newtheorem{lem}{Lemma}[section]

\newtheorem{remark}{Remark}[section]

\usepackage{graphicx}
\usepackage{fancybox}

\usepackage[latin1]{inputenc}

\usepackage{tcolorbox}

\def\uphrulefill{%
  \textcolor{darkblue!70}{\leavevmode\leaders\hrule height 1.5pt\hfill\kern0pt \rule[-0.56ex]{1.5pt}{5pt}}\\}
  
  \def\downhrulefill{%
  \textcolor{darkblue!70}{\leavevmode\leaders\hrule height 1.5pt\hfill\kern0pt \rule[0.0ex]{1.5pt}{5pt}}\\}

\usepackage{graphicx,fancybox,bbding,manfnt,wrapfig,multicol,mathrsfs}

\usepackage{amsmath,amssymb,color,makeidx,rotating,pifont}

\usepackage{mathtools}

\mathtoolsset{showonlyrefs=true}  

\begin{document}

\LARGE\sffamily 
\begin{center}\bfseries
Beta-weighted non-local differential operators and related stochastic processes
\end{center}

\normalsize

\begin{center}
{Luisa Beghin}$^{\textrm{a}}$\footnote{Corresponding author.}
,
{Nikolai Leonenko}$^{\textrm{b}}$, {Thomas Simon}$^{\textrm{c}}$ and
{Jayme Vaz}$^{\textrm{d}}$ 

\footnotesize{
		$$\begin{tabular}{llll}
$^{a}$Department of Statistical Sciences, Sapienza University, p.le Aldo Moro 5, Rome, 00185, Italy. \\
$^{b}$Cardiff School of Mathematics, Cardiff University, Senghennydd Road,
Cardiff, CF24 4AG, UK.\\
$^{c}$Laboratoire Paul Painlev\'e, Universit\'e de Lille, 42 rue Paul Duez, 59000 Lille, France.\\
$^{d}$Departamento de Matem\'atica Aplicada, 
Universidade Estadual de Campinas, 13087-859, Campinas, SP, Brazil. 
\end{tabular}$$}
\end{center}

\normalsize\rmfamily\sffamily

\bigskip

\small\rmfamily

\begin{center}
\begin{minipage}[c]{14cm}
\centerline{\bfseries\sffamily Abstract}
\smallskip
\begin{abstract}
In this work we introduce a class of non-local differential operators defined through a beta-weighted averaging of the ordinary derivative. We investigate their analytical properties and establish connections with the Caputo and Erd\'elyi-Kober operators. Differential equations involving the beta-weighted derivative are studied by Mellin transform methods, leading to solutions represented in terms of Barnes $G$-functions and a new class of $G$-hypergeometric functions. We also analyze asymptotic properties, Laplace transforms, and the second-order equation involving the sequential beta-weighted derivative.  Finally, we present stochastic applications of these results, showing that continuous-time random walks, with waiting times characterized by the beta-weighted derivative, converge to Brownian motions time-changed by a scaled inverse stable subordinator. We compare this anomalous-diffusion model with a time-changed Brownian motion whose one-dimensional distribution solve a heat-type equation with beta-weighted derivative.  
\end{abstract}
\end{minipage}
\end{center}

\normalsize\rmfamily

\medskip

\DeclareGraphicsExtensions{.gif,.pdf}

\normalsize


\section{Introduction}

\noindent
Fractional derivatives, in particular those in the sense of Caputo and 
Riemann--Liouville, play a central role in fractional calculus and its 
applications. The Caputo derivative is particularly very useful in 
modeling because of its connection with 
initial value problems \cite{Kilbas,Podlubny,Diethelm}.  
In contrast with classical derivatives, which are local in time, both the 
Caputo and Riemann-Liouville derivatives are defined through convolution-type 
integrals involving the past history of the function. This makes them 
natural candidates for describing phenomena in which the present state depends 
not only on the current configuration but also on the cumulative effect of 
previous states. Typical examples arise in viscoelasticity, where stress-strain relations 
exhibit memory effects that cannot be captured by integer-order models, and in 
anomalous diffusion processes, where the mean square displacement scales as a 
non-linear power of time. 

From a probabilistic viewpoint, fractional derivatives are closely related to 
continuous-time random walks with heavy-tailed waiting time distributions and 
to subordinated stochastic 
processes \cite{MontrollWeiss1965,MetzlerKlafter2000,MeerschaertSikorskii2012}. 
In particular, they naturally describe 
dynamics governed by inverse stable subordinators, leading to time-fractional 
evolution equations. This connection explains the appearance of fractional 
operators in models of transport in disordered media, trapping phenomena, and 
systems with long-time correlations.

Despite their success, fractional derivatives raise certain issues when examined 
from the viewpoint of dimensional analysis \cite{VazCapelas}. 
Let $f(t)$ be a quantity with physical dimension $[F]$, where $t$ has dimension 
$[T]$, and for definiteness we shall regard $t$ as time. 
The first-order derivative $df/dt$ has physical dimension 
$\left[df/dt\right] = [T]^{-1}[F]$, 
which is consistent with the interpretation of a rate of change per unit time. 
In contrast, a fractional derivative of order $\alpha \in (0,1)$, whether in 
the Caputo or in the Riemann--Liouville sense, formally has dimension
$\left[\mathcal{D}_t^\alpha f\right] = [T]^{-\alpha}[F]$. 
This leads to a first difficulty, i.e., the loss of direct dimensional compatibility 
with classical physical quantities. In an equation of the form
$\mathcal{D}_t^\alpha f(t) = g(t)$, 
the right-hand side must have dimension $[T]^{-\alpha}[F]$, which generally 
differs from the dimension associated with the corresponding classical model 
(where $\alpha = 1$). 

A second issue is the implicit dependence on a characteristic time scale. 
To restore dimensional consistency, it is often necessary to introduce a 
parameter $\tau$ with dimension of time and consider instead an equation of the form
$\tau^{1-\alpha} \mathcal{D}_t^\alpha f(t) = g(t)$. 
This shows that fractional derivatives are not intrinsically scale-free: 
their use in physical models typically requires the introduction of an 
external time scale, which does not arise in the first-order case.

These features also affect the interpretation of model parameters. For instance, 
replacing the classical diffusion equation
\[
\frac{\partial u}{\partial t} = D \nabla^2 u
\]
by its fractional counterpart
\[
\mathcal{D}_t^\alpha u = D \nabla^2 u
\]
implies that the diffusion coefficient $D$ acquires dimension $[L]^2 [T]^{-\alpha}$, 
rather than the standard $[L]^2 [T]^{-1}$. Thus, the parameter $D$ no longer 
represents a diffusion constant in the classical sense, and its physical 
interpretation must be reconsidered. 

One way to circumvent these issues, discussed in \cite{VazCapelas}, is through 
the introduction of the \textit{dimensional regularized} Caputo derivative $\mathsf{D}_t^\dagger$ 
defined as 
\begin{equation}
\mathsf{D}_t^\dagger = t^{\alpha-1} \sideset{_{\scriptscriptstyle{\textrm C}}^{}}{_t^{\alpha}}{\operatorname{\mathit D}} , 
\end{equation} 
where $\sideset{_{\scriptscriptstyle{\textrm C}}^{}}{_t^{\alpha}}{\operatorname{\mathit D}}$ denotes the Caputo fractional derivative operator. 
Also in \cite{VazCapelas} some fractional 
differential equations involving this dimensional regularized
Caputo derivative have been solved and compared with the
solutions from equations involving the standard Caputo fractional derivative. 

Although the dimensional regularized Caputo derivative was introduced in the context 
of the dimensional analysis of fractional differential equations, there is another 
very interesting interpretation for it.
Actually, let $K : (0,\infty) \to \mathbb{R}_+$ be a memory kernel, and define
\begin{equation}
\label{memory.operator}
\mathcal{M}_{\scriptscriptstyle K} f(t)
=
\int_0^t K(t-\tau)\, f^\prime(\tau)\, d\tau.
\end{equation}
This operator can be interpreted as a cumulative contribution of past variations, 
where the kernel $K$ determines how strongly each past instant influences the 
present state. The convolution-type nonlocal operators \eqref{memory.operator}, with kernels generated by the tails of L\'evy measures of subordinators, have been introduced and studied by Kochubei \cite{Kochubei2011} and Toaldo \cite{Toaldo2015}, among others. We adopt a different approach.

If $K$ is integrable on $(0,t)$, one may introduce the normalized density
\begin{equation}
\label{prob.density}
p_t(\tau)
=
\frac{K(t-\tau)}{\displaystyle \int_0^t K(t-\tau^\prime)\,d\tau^\prime},
\qquad 0 \leq \tau \leq t,
\end{equation}
and interpret \eqref{memory.operator} as
\begin{equation}
\label{expectation.form}
\mathcal{M}_{\scriptscriptstyle K} f(t)
=
\left(\int_0^t K(t-\tau^\prime)\,d\tau^\prime\right)
\mathbb{E}\bigl[f^\prime(T_t)\bigr],
\end{equation}
where $T_t$ is a random variable with density $p_t$. In this form, the operator 
$\mathcal{M}_{\scriptscriptstyle K}$ appears as a rescaled expectation of the ordinary first-order derivative 
evaluated at a random past time. This provides a natural probabilistic 
interpretation: the system evolves according to an effective rate obtained by 
sampling its past history with a prescribed distribution.

A particularly important case arises when the kernel is a power law,
\begin{equation}
K(r) = \frac{r^{-\alpha}}{\Gamma(1-\alpha)}, 
\qquad 0 < \alpha < 1.
\end{equation}
In this case, the operator \eqref{memory.operator} coincides with the 
Caputo fractional derivative,
\begin{equation} 
\label{Caputo.def}
\sideset{_{\scriptscriptstyle{\textrm C}}^{}}{_x^{\alpha}}{\operatorname{\mathit D}} f(t)
=
\frac{1}{\Gamma(1-\alpha)}
\int_0^t (t-\tau)^{-\alpha} f^\prime(\tau)\,d\tau.
\end{equation}
A direct calculation shows that
$$
\int_0^t K(t-\tau^\prime)\, d\tau^\prime = \frac{1}{\Gamma(1-\alpha)} 
\int_0^t (t-\tau^\prime)^{-\alpha}\,d\tau^\prime
=
\frac{t^{1-\alpha}}{\Gamma(2-\alpha)},
$$
and from \eqref{expectation.form}, 
\begin{equation}
\label{caputo.expectation}
\sideset{_{\scriptscriptstyle{\textrm C}}^{}}{_x^{\alpha}}{\operatorname{\mathit D}} f(t)
=
\frac{t^{1-\alpha}}{\Gamma(2-\alpha)}
\,\mathbb{E}\bigl[f^\prime(T_t)\bigr],
\end{equation}
where $T_t$ is a random variable on $[0,t]$ with density proportional to 
$(t-\tau)^{-\alpha}$. As a consequence, 
\begin{equation}
\label{caputo.dim.regularized.expectation}
\mathsf{D}_t^\dagger f(t) = \frac{1}{\Gamma(2-\alpha)} \mathbb{E}\bigl[f^\prime(T_t)\bigr].
\end{equation}

Thus we see that the dimensional regularized Caputo derivative can be interpreted as 
a rescaled expectation of the ordinary first-order derivative 
evaluated at a random past time. 
This interpretation clarifies both the origin of non-locality 
and the anomalous dimensional behavior of fractional operators, and suggests 
natural generalizations obtained by replacing the underlying probability 
distribution, which opens the possibility of modeling different types of memory mechanisms. 

The natural generalization of a probability distribution proportional to 
$(t-\tau)^{-\alpha}$ with $\tau\in [0,t]$ is the \textit{beta distribution}. The aim of this work 
is to introduce this new type of derivative involving the beta distribution 
and investigate its applications. 

This work is organized as follows. In Section~2 we introduce the beta-weighted 
derivative and discuss its main analytical properties. We show how this operator 
naturally arises from a probabilistic averaging procedure based on the beta 
distribution and discuss its connections 
with the Caputo derivative and the Erd\'elyi-Kober
fractional integral and derivative. We also derive its Mellin transform.
In Section~3 we study the first-order differential equation associated with 
the beta-weighted derivative. By applying Mellin transform techniques, we obtain 
explicit solutions in terms of a new class of special functions represented through 
Barnes $G$-functions and Mellin--Barnes integrals. Motivated by this representation, 
we introduce the so-called $G$-hypergeometric functions and investigate their 
convergence properties. We further establish the complete monotonicity of the 
solutions, derive their asymptotic expansions and obtain representations for their 
Laplace transforms together with their asymptotic behavior.
Section~4 is devoted to second-order differential equations involving sequential 
beta-weighted derivatives. We analyze the corresponding homogeneous equation and 
derive explicit solutions in terms of the functions introduced previously. 
In Section~5 we present stochastic applications of the theory developed in the 
previous sections. We construct continuous-time random walks with waiting-time 
distributions governed by the beta-weighted relaxation equation and study their 
asymptotic scaling limits. In particular, we establish convergence results toward 
subordinated Brownian motions and discuss the relation with anomalous diffusion and 
fractional heat equations. We also introduce a class of time-changed Brownian 
motions whose one-dimensional distributions satisfy diffusion equations involving 
the beta-weighted derivative. Finally, in the appendices, we collect several 
auxiliary results concerning the Barnes $G$-function, Mellin--Barnes integrals, 
residue calculations, and asymptotic estimates used throughout the paper.

\section{The beta-weighted derivative}

Let $(\Omega,\mathcal{F},P)$ be a probability space and let $T_t$ be a random variable
taking values in $(0,t)$, for $t \in \mathbb{R}^+$. Assume that $T_t$ has probability density
$p_{(0,t)}^{(a,b)}(\tau)$ given by
\begin{equation}
\label{beta.pdf}
p_{(0,t)}^{(a,b)}(\tau) =
\begin{cases}
\displaystyle 
\frac{\tau^{a-1}(t-\tau)^{b-1}}{t^{a+b-1} B(a,b)} ,
& \tau \in (0,t), \\[1ex]
0 , & \text{otherwise},
\end{cases}
\end{equation}
where $a,b>0$. This is the beta probability density on the interval $(0,t)$.

Let $f:\mathbb{R}^+ \to \mathbb{R}$ be a continuously differentiable function. 
The expected value of $f(T_t)$ is given by
\begin{equation}
\mathbb{E}_{(0,t)}^{(a,b)}(f) 
= \mathbb{E}\big(f(T_t)\big)
= \int_0^t p_{(0,t)}^{(a,b)}(\tau) f(\tau) \, d\tau .
\end{equation}

We now introduce parameters $\alpha \in (0,1)$ and $\nu>-1$ such that
\begin{equation}
a = 1 + \nu , \qquad b = 1 - \alpha .
\end{equation}
We define the operator $(\mathbb{D}_t^{(\alpha,\nu)} f)(t)$ as the expected value
of the derivative $f^\prime(T_t)$, namely,
\begin{equation}
\label{mathcal.D.0}
\mathbb{D}_t^{(\alpha,\nu)}f 
= \mathbb{E}_{(0,t)}^{(\nu+1,1-\alpha)}\big(f^\prime(T_t)\big)
= \frac{t^{\alpha-\nu-1}}{B(1-\alpha,1+\nu)} 
\int_0^t \tau^\nu (t-\tau)^{-\alpha} f^\prime(\tau) \, d\tau .
\end{equation}

It is convenient to introduce, in analogy with
eq.\eqref{caputo.dim.regularized.expectation}, the following rescaled operator.

\begin{definition}
\label{def:beta-weighted-derivative}
Let $\alpha \in (0,1)$ and $\nu>-1$. For a differentiable function 
$f:\mathbb{R}^+ \to \mathbb{R}$, the beta-weighted derivative
of $f$ is defined by
\begin{equation}
\label{mathcal.D.2}
(\mathscr{D}_t^{(\alpha,\nu)} f)(t) := 
\frac{t^{\alpha-\nu-1}}{\Gamma(1-\alpha)} 
\int_0^t \tau^\nu (t-\tau)^{-\alpha} f^\prime(\tau) \, d\tau .
\end{equation}
\end{definition}

The operators $\mathbb{D}_t^{(\alpha,\nu)}$ and $\mathscr{D}_t^{(\alpha,\nu)}$ 
are related by
\begin{equation}
\label{mathcal.D.1}
\mathbb{D}_t^{(\alpha,\nu)} f 
= \frac{\Gamma(2+\nu-\alpha)}{\Gamma(1+\nu)} 
\, \mathscr{D}_t^{(\alpha,\nu)} f .
\end{equation}

The operator $\mathscr{D}_t^{(\alpha,\nu)}$ can be interpreted as a 
non-local differential operator obtained by a beta-weighted averaging 
of the ordinary first-order derivative over the interval $(0,t)$. In particular, 
it provides a probabilistic interpretation of non-local operators, 
where the memory kernel arises naturally from the beta distribution.

As an example, from eq.\eqref{mathcal.D.2} it follows that 
\begin{equation}
\label{eq.derivative.rho}
\mathscr{D}_t^{(\alpha,\nu)}t^\rho = 
\begin{cases} \displaystyle \frac{\rho \Gamma(\nu+\rho)}{\Gamma(\nu+\rho+1-\alpha)} 
t^{\rho - 1} , & \; \rho \neq 0 , \\[1ex]
0 , & \; \rho = 0 , 
\end{cases}
\end{equation}
with $1-\alpha > 0$ and $\nu +\rho > -1$ to guarantee the 
convergence of the integral in eq.\eqref{mathcal.D.2}.

\subsection{Connections with Caputo and Erd\'elyi-Kober operators}

Next we establish the analytical connections between the beta-weighted derivative $\mathscr{D}_t^{(\alpha,\nu)}$ and known differential and fractional operators. We identify its classical limit, its relation with the Caputo derivative, its representation via Erd\'elyi-Kober (EK) fractional integrals, and the corresponding inverse operator. We also compare it with Caputo-type modifications of 
EK derivatives, with emphasis on scaling properties.

\paragraph{(i) Classical limit.} The Gelfand-Shilov distribution is defined as  
\begin{equation}
G_\nu(t) = \begin{cases} \displaystyle \frac{t^{\nu-1}}{\Gamma(\nu)}H(t) , & \quad \nu > 0 , \\
G_{\nu+1}^\prime(t) , & \quad \nu \leq 0 . 
\end{cases} 
\end{equation}
and is such that 
\begin{equation}
\lim_{\nu\to 0} G_\nu(t) = \delta(t)   
\end{equation}
where  $\delta(t)$ is the Dirac delta function.  Note that 
\begin{equation}
p^{(a,b)}_{(0,t)}(\tau) = \frac{\Gamma(a+b)}{t^{a+b-1}} G_a(t) G_b(t-\tau) , 
\end{equation} 
and then 
\begin{equation}
\lim_{\substack{a\to 1\\ b\to 0}} p_{(0,t)}^{(a,b)}(\tau) = \delta(t-\tau) , \
\qquad
\lim_{\substack{a\to 0\\ b\to 1}} p_{(0,t)}^{(a,b)}(\tau) = \delta(\tau) .
\end{equation}
So we have
\begin{equation}
\lim_{\substack{\alpha\to 1\\ \nu \to 0}} \mathscr{D}_t^{(\alpha,\nu)} f = 
f^\prime(t) . 
\end{equation}
The same holds, of course, for $\mathbb{D}_t^{(\alpha,\nu)}f$. 

\medskip

\paragraph{(ii) Caputo-type reduction.} If $\alpha < 1$ and $\nu = 0$ we have 
\begin{equation}
\mathscr{D}_t^{(\alpha,0)} f = t^{\alpha-1} \, \sideset{_{\scriptscriptstyle{\textrm C}}^{}}{_t^{\alpha}}{\operatorname{\mathit D}} f
\end{equation}
which is the dimensional regularized Caputo derivative \cite{VazCapelas}. 
It can also be seen as a particular case of the stretched derivative 
$\mathcal{D}^{(\alpha,\rho)}_t f = t^{-\rho} \sideset{_{\scriptscriptstyle{\textrm C}}^{}}{_t^{\alpha}}{\operatorname{\mathit D}}[f(t)]$ considered in \cite{BLV} with $\rho = 1-\alpha$. 

\medskip

\paragraph{(iii) EK representation and inverse.} We also note that 
\begin{equation}
\mathscr{D}_t^{(\alpha,\nu)} f = 
\mathfrak{I}_t^{(1-\alpha,\nu)} f^\prime,  \label{EKint}
\end{equation}
where 
$$
\mathfrak{I}_t^{(\mu,\nu)} f = 
\frac{t^{-(\mu+\nu)}}{\Gamma(\mu)} \int_0^t \tau^\nu (t- \tau)^{\mu-1}\, f(\tau)\, d\tau 
$$
is the EK fractional integral of order $\mu$. We will assume 
$\nu > -1$. Let us explore this fact a little further.

In \cite{LUC} the authors defined (adapting their notation 
to ours) the right-hand sided EK fractional integral of orders $\mu$ and $\nu$ by 
\begin{equation}\label{EKint2}
\left(I_\beta^{\nu,\mu} f\right)(t) = 
\frac{\beta}{\Gamma(\mu)} t^{-\beta(\nu+\mu)} 
\int_0^t (t^\beta-\tau^\beta)^{\mu-1} \tau^{\beta(\nu+1)-1} f(\tau)\, d\tau ,
\end{equation}
with $\mu,\beta > 0$ and $\nu \in \mathbb{R}$. We therefore have a slight difference in notation, namely,
$$
\mathfrak{I}_t^{(\mu,\nu)} = I_1^{(\nu,\mu)} . 
$$
The right-hand sided EK fractional derivative of order $\mu$ 
($n-1 < \mu \leq n$, $n\in \mathbb{N}$), is defined as 
\begin{equation}
\left(D_\beta^{\nu,\mu}f\right)(t) = 
{\displaystyle \prod_{j=1}^n} \left(\nu+j + \frac{1}{\beta} t \frac{d\;}{dt}\right) 
\left(I_\beta^{\nu+\mu,n-\mu}f \right)(t) . \label{EKder}
\end{equation}

It is proved in \cite{LUC} that the EK derivative \eqref{EKder} is the left-inverse operator of the EK integral \eqref{EKint2}, in the space $C_\alpha$, where the latter denotes the space of all functions $f(\cdot)$ 
such that $f(t)=t^pf_1(t),$ for $p>\alpha$, $f_1 \in C[0,+\infty)$. Thus 
\begin{equation}
    D_\beta^{\nu,\mu}I_\beta^{\nu,\mu} f(t) \equiv f(t), \qquad f \in C_\alpha, \alpha>\beta(\gamma+1).
\end{equation}

By recalling the relationship of the beta-weighted derivative with the EK integral, given in \eqref{EKint}, we can write the left-inverse operator of $\mathscr{D}_t^{(\alpha,\nu)}$ as follows:
$$
\mathscr{I}_t^{(\alpha,\nu)}f:=\int_0^t \left(D_1^{\nu,1-\alpha}f\right)(t)dt,
$$
since 
$$
\mathscr{I}_t^{(\alpha,\nu)}\mathscr{D}_t^{(\alpha,\nu)}f(t)=\int_0^t D_1^{\nu,1-\alpha}I_1^{\nu,1-\alpha} f^\prime(\tau)d\tau=f(t)-f(0),\qquad t >0.
$$

\paragraph{(iv) Caputo-type EK derivative and scaling.} The right-hand sided Caputo-type modification of the EK fractional derivative 
of order $\mu$ is defined as 
\begin{equation}
\sideset{_{\ast}^{}}{_{\beta}^{\nu,\mu}}{\operatorname{\mathit{D}}} f(t) = 
\left( I_\beta^{\nu+\mu,n-\mu} 
{\displaystyle \prod_{k=0}^{n-1}}\left(1+\nu+k+\frac{1}{\beta} t\frac{d\;}{dt}\right) 
f\right)(t) . 
\end{equation}

For $n = 1$ and $\beta=1$ the above expression reduces to  
\begin{equation}
\sideset{_{\ast}^{}}{_{}^{\nu,\mu}}{\operatorname{\mathit{D}}} f(t) = 
\frac{1}{\Gamma(\mu)} t^{-(\nu+\mu)} 
\int_0^t (t-\tau)^{\mu-1} \tau^{\nu}
\bigg( 1+\nu+ \tau\frac{d\;}{d\tau}\bigg) f (\tau) \, d\tau , 
\end{equation}
and denoting 
$$
\mu-1 = -\alpha
$$
with $0 < \alpha < 1$, we have
\begin{equation}
\label{EKCaputo}
\sideset{_{\ast}^{}}{_{}^{\nu,1-\alpha}}{\operatorname{\mathit{D}}} f(t) = 
\frac{1}{\Gamma(1-\alpha)} t^{-(\nu+1-\alpha)} 
\int_0^t (t-\tau)^{-\alpha} \tau^{\nu}
\bigg( 1+\nu+ \tau\frac{d\;}{d\tau}\bigg) f (\tau) \, d\tau . 
\end{equation}
Note the differences between eq.\eqref{EKCaputo} and eq.\eqref{mathcal.D.0} or 
eq.\eqref{mathcal.D.2}. In particular, the operator $\sideset{_{\ast}^{}}{_{}^{\nu,1-\alpha}}{\operatorname{\mathit{D}}}$ is \textit{scale-invariant}, that is, if $t \mapsto \lambda t$ ($\lambda > 0$), then $\sideset{_{\ast}^{}}{_{}^{\nu,1-\alpha}}{\operatorname{\mathit{D}}} f 
\mapsto \sideset{_{\ast}^{}}{_{}^{\nu,1-\alpha}}{\operatorname{\mathit{D}}} f$, while the operator 
$\mathscr{D}_t^{(\alpha,\nu)}$ is not, that is, if $t \mapsto \lambda t$ then 
$\mathscr{D}_{\lambda t}^{(\alpha,\nu)} f = \lambda^{-1} \mathscr{D}_t^{(\alpha,\nu)}$.  
Thus $\sideset{_{\ast}^{}}{_{}^{\nu,1-\alpha}}{\operatorname{\mathit{D}}}$ is 
a generalization of the differential operator $t\frac{d\;}{dt}$, while 
$\mathscr{D}_t^{(\alpha,\nu)}$ is a generalization of $\frac{d\;}{dt}$. 



\subsection{Mellin transform of the Beta-weighted derivative}

Let $\mathcal{M}[f(t),s]$ be the Mellin transform
$$
\mathcal{M}[f(t);s] = F(s) = \int_0^\infty f(t) t^{s-1}\, dt . 
$$
We want to show the following
\begin{lem}
\label{lemma.mellin}
Let $0<\alpha<1$ and $\nu>-1$. Suppose that $f$ is such that the Mellin transforms exist, the boundary term in the Mellin transform of $f^\prime$ vanishes, and the Mellin convolution theorem applies. Then
\begin{equation}
\label{mellin.derivative}
\mathcal{M}[\mathscr{D}_t^{(\alpha,\nu)} f;s]
=
 \frac{(1-s)\Gamma(1+\nu-s)}{\Gamma(2+\nu-\alpha-s)}F(s-1). 
\end{equation}
\end{lem}

\begin{proof} If the boundary term in the Mellin transform of 
$f^\prime$ vanishes we have  (\cite{Podlubny})
$$
\mathcal{M}[f^\prime(t);s] = (1-s)F(s-1) , 
$$
and the Mellin convolution theorem asserts that 
$$
\mathcal{M}\left[\int_0^\infty g(\tau) f(t \tau)\, d\tau\right] = F(s) G(1-s) , 
$$
where $G(s) = \mathcal{M}[g(t);s]$. 
Now using $\tau = t u$ in eq.\eqref{mathcal.D.2} we have 
$$
\mathscr{D}_t^{(\alpha,\nu)} f   =  
\frac{1}{\Gamma(1-\alpha)} 
\int_0^1 u^\nu (1-u)^{-\alpha} f^\prime(t u) \, du =
\int_0^\infty g(u) f^\prime(t u)\, d u ,
$$
with 
$$
g(t) = \frac{t^\nu (1-t)^{-\alpha}}{\Gamma(1-\alpha)} H(1-t) ,
$$
where $H(\cdot)$ is the Heaviside step function. The Mellin 
transform of $g(t)$ is 
$$
\mathcal{M}[g(t);s] = 
\int_0^1 \frac{ t^{\nu+s-1} (1-t)^{-\alpha}}{\Gamma(1-\alpha)} \, dt = 
\frac{ B(\nu+s,1-\alpha)}{\Gamma(1-\alpha)} = \frac{\Gamma(\nu+s)}{\Gamma(1+\nu-\alpha+s)} .
$$
Using, from the convolution theorem, 
\begin{equation}
\mathcal{M}[\mathscr{D}_t^{(\alpha,\nu)} f ;s] = 
 \frac{(1-s)\Gamma(1+\nu-s)}{\Gamma(2+\nu-\alpha-s)}F(s-1). 
\end{equation}
\end{proof}

\section{First order differential equation}
\label{section.6}

Let us consider the differential equation 
\begin{equation}
\label{mathcal.eq.1}
\mathscr{D}_t^{(\alpha,\nu)} f = - \kappa f, 
\end{equation}
which corresponds to 
\begin{equation}
\label{mathcal.eq.2}
\mathbb{D}_t^{(\alpha,\nu)} f = - \kappa^\prime f ,
\end{equation}
where 
$$
\kappa^\prime = \frac{\Gamma(2+\nu-\alpha)}{\Gamma(1+\nu)}  \kappa .
$$

We are now able to present the main result of this section.

\begin{theorem}\label{main}
    The solution to equation \eqref{mathcal.eq.1}, with initial condition $f(0)=1$, is given by
\begin{equation}
\label{eq.solution}
f(t) = \Psi_\alpha^\nu(-\kappa t)  := 
\frac{G(1+\nu)}{G(2-\alpha+\nu)}\sum_{n=0}^\infty   \frac{G(2+n-\alpha+\nu)}{G(1+n+\nu)} \frac{(-\kappa t)^n}{n!} ,
\end{equation}
where $G(\cdot)$ is the Barnes $G$-function (see Appendix~\ref{appendix.Barnes}).    
\end{theorem}

\begin{proof}
Taking the 
Mellin transform of eq.\eqref{mathcal.eq.1} and using Lemma~\ref{lemma.mellin} 
we obtain 
\begin{equation}
\label{eq.S}
\frac{F(s)}{F(s-1)} = -\frac{(1-s)}{\kappa} \frac{\Gamma(1+\nu-s)}{\Gamma(2+\nu-\alpha-s)} . 
\end{equation}
Using $\Gamma(z+1) = z \Gamma(z)$ and eq.\eqref{barnes.g}, we can write $F(s)$ as 
\begin{equation}
F(s) = f_0 \kappa^{-s} \Gamma(s) \frac{G(2+\nu-\alpha-s)}{G(1+\nu-s)} , \label{mel} 
\end{equation}
where $f_0$ is arbitrary.  
Using the inversion formula for the Mellin transform, we have
\begin{equation}
\label{inverse.mellin}
f(t) = \frac{f_0}{2\pi i} \int_{c-i\infty}^{c+i\infty} (\kappa t)^{-s} 
\frac{\Gamma(s)G(2+\nu-\alpha-s)}{G(1+\nu-s)} ds .
\end{equation}

Next we will evaluate the integral in eq.\eqref{inverse.mellin}. 
The poles of the integrand are the ones of 
$\Gamma(s)$, that is, $s_n = -n$ ($n=0,1,2,\ldots$), and 
the ones of $1/G(1+\nu-s)$, that is, $1+\nu-s_m = -m$ ($m=0,1,2,\ldots$), i.e., 
$s_m = m +1 +\nu$ ($m=0,1,2,\ldots$). Since we are assuming $\nu > -1$,   
if we choose $c$ such that $0 < c < \zeta$  with $\zeta = 1+\nu > 0$,  then the poles of $\Gamma(s)$ 
will be located on the left and the poles of $1/G(1+\nu-s)$ will be located 
on the right of the vertical line $(c-i\infty,c+i\infty)$.  

In order to analyse the convergence of this integral, we will consider 
the limit $N \to \infty$ for the integral along the line segment $(c-i(N+1/2),c+i(N+1/2))$ 
(the dashed line segment in Figure~\ref{fig.2}). 
Moreover, since the integrand is a holomorphic function in the 
entire plane except at the poles above, we can deform the line segment  
$(c-i(N~+~1/2), c+i(N+1/2))$  into the contour $C_N^- \, \cup \,  C_\epsilon  \, \cup \,  C_N^+ $, where 
$C_N^- = (-i(N~+~1/2),-\epsilon)$, $C_N^+ = (\epsilon,i(N+1/2))$ and 
$C_\epsilon$ is the arc $\{s \in \mathbb{C}\,|\,
|s| = \epsilon, -\pi/2 \leq \operatorname{arg}(s) \leq \pi/2\}$ 
encircling the pole $s=0$ from the right (see Figure~\ref{fig.2}) . 

\begin{figure}[h]  
\begin{center}
\begin{tikzpicture}[scale=1]
\draw[->] (-6,0) -- (2,0) node[below] {${\scriptscriptstyle \operatorname{Re}s}$};
\draw[->] (0,-4.7) -- (0,4.8) node[right] {${\scriptscriptstyle \operatorname{Im}s}$};
\draw[orange,fill=orange] (0,0) circle (1.2pt);
\draw[orange,fill=orange] (-1,0) circle (1.2pt);
\draw[orange,fill=orange] (-2,0) circle (1.2pt);
\draw[orange,fill=orange] (-4,0) circle (1.2pt);
\draw[orange,fill=orange] (-5,0) circle (1.2pt);
\draw[orange,fill=orange] (0.8,0) circle (1.2pt);
\draw[-] (1,0.05) -- (1,-0.05);
\node[below] at (-1,0) {${\scriptscriptstyle -1}$};
\node[below right] at (1,0) {${\scriptscriptstyle 1}$};
\node[below left] at (0,0) {${\scriptscriptstyle 0}$};
\node[below] at (-2,0) {${\scriptscriptstyle -2}$};
\node[below] at (-3,0) {$\cdots$};
\node[below] at (-4,0) {${\scriptscriptstyle -N}$};
\node[below] at (-5.2,0) {${\scriptscriptstyle -(N+1)}$};
\node[below] at (0.8,0) {${\scriptscriptstyle \zeta}$};
\draw[cyan,very thick] (0,0.3) -- (0,4.5);
\draw[very thick,dashed] (0.55,-4.5) -- (0.55,4.5);
\draw[cyan,very thick] (0,-0.3) -- (0,-4.5);
\draw[->,>=triangle 45,cyan,very thick] (0,1) -- (0,1.9);
\draw[cyan, very thick] (0,4.5) arc (90:270:4.5);
\draw[cyan, very thick,->,>=triangle 45] (0,4.5) arc (90:160:4.5);
\draw[cyan, very thick] (0,-0.3) arc (-90:90:0.3);
\node[below] at (0.3,-0.3) {\footnotesize ${\textcolor{cyan}{C_\epsilon }}$};
\node[left] at (0,3) {\footnotesize ${\textcolor{cyan}{C_N^+ }}$};
\node[left] at (0,-3) {\footnotesize ${\textcolor{cyan}{C_N^-}}$};
\node[right] at (0.8,-4.5) {$c-i(N+1/2)$};
\node[right] at (0.8,4.5) {$c+i(N+1/2)$};
\node[below, cyan] at (-5,2.6) {$C_N$};
\draw[->] (0,0) -- (-3.15,3.15);
\node[right] at (-2,2) {${\scriptscriptstyle N+1/2}$};
\end{tikzpicture}
\end{center}
\caption{Integration contour for eq.\eqref{inverse.mellin}.\label{fig.2}}
\end{figure}
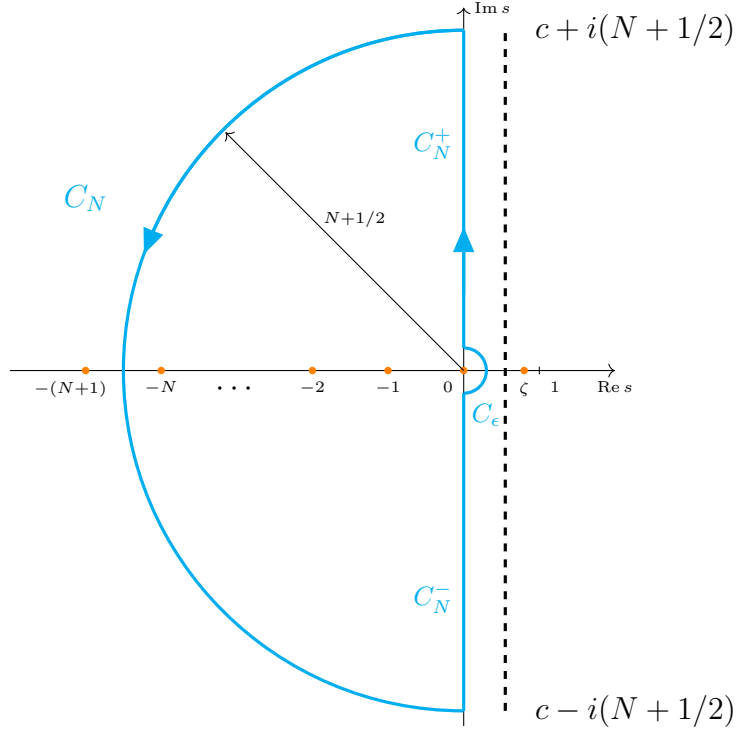

Since the integral along the arc $C_\epsilon$ vanishes for $\epsilon \to 0$, 
we only need to analyze the integrals along $C_N^-$ and $C_N^+$. We can 
write $s \in C_N^\pm$ as $s = R\operatorname{e}^{i\theta}$ with 
$\epsilon \leq R \leq (N+1/2)$ and $\theta = \pi/2$ for $s \in C_N^+$ and
$\theta = -\pi/2$ for $s \in C_N^-$. Let us analyze 
\begin{equation}
\label{log.integrand}
\log|J(t,s)| =  \log\left|(\kappa t)^{-s} 
\frac{\Gamma(s)G(2+\nu-\alpha-s)}{G(1+\nu-s)} \right| . 
\end{equation}
We have  
$$
\log|(\kappa t)^{-s}| = -R\cos\theta \log|\kappa t| + 
R\sin\theta \operatorname{arg}(\kappa t) , 
$$
and from Stirling formula, 
$$
\log|\Gamma(s)| = \left(R\cos\theta -\frac{1}{2}\right) \log{R} - 
R(\theta\sin\theta + \cos\theta) + \mathcal{O}(1) . 
$$
From the asymptotic formula for $\log G(s)$ given in \cite{Ferreira},
we have 
\begin{equation*}
\begin{aligned} 
& \log|G(1+a-s)| = -\frac{3}{4} R^2\cos{2\theta} + a R \cos\theta - 
\frac{\log{2\pi}}{2} R\cos\theta + \frac{1}{2}R^2 \log{R} \cos{2\theta}\\[1ex]
& \phantom{\log|G(1+a-s)| =} -\frac{1}{2}R^2 \theta \sin{2\theta} + 
\frac{\pi}{2} R^2 \sin{2|\theta|} - a R \log{R}\cos\theta + a R \theta\sin\theta\\[1ex]
& \phantom{\log|G(1+a-s)| =} -a\pi R \sin|\theta| + \left(\frac{a^2}{2}-\frac{1}{12}\right) \log{R} + \mathcal{O}(1) .
\end{aligned}
\end{equation*}
Collecting these above expressions in eq.\eqref{log.integrand} we obtain 
\begin{equation}
\label{exp.log.integrand}
\begin{aligned}
\log\left|(\kappa t)^{-s} \frac{\Gamma(s) G(2-s-\alpha)}{  G(1-s)}\right|  = & 
\alpha\cos\theta \, R\log{R} + A R \\[1ex] 
& + \frac{(\alpha+\nu)(\alpha+\nu-2)}{2}\log{R} + \mathcal{O}(1) ,
\end{aligned}
\end{equation}
where 
$$
A = -\cos\theta \log|\kappa t| + \sin\theta \operatorname{arg}(\kappa t) 
-\alpha(\theta\sin\theta + \cos\theta) -(1-\alpha)\pi \sin|\theta| .
$$
For $s \in C_N^\pm$ (i.e., $\theta = \pm \pi/2$) the leading 
term in eq.\eqref{exp.log.integrand} is  
$$
AR =  
\left[\pm \operatorname{arg}(\kappa t) + \left(\frac{\alpha}{2}-1\right)\pi\right] R 
$$
where the $\pm$ signs corresponds to $\theta = \pm \pi/2$. 
Thus, the integral will converges for $R \to \infty $ if 
$$
\pm \operatorname{arg}(\kappa t) + \left(\frac{\alpha}{2}-1\right)\pi < 0 ,
$$
that is, 
$$
|\operatorname{arg}(\kappa x)| < \left(1-\alpha/2\right)\pi . 
$$
Since $0 < \alpha < 1$, this means that the integral will converges 
if $\operatorname{Re}(\kappa t) > 0$. 

Now let us consider the closed contour $C_N^- \, \cup \,  C_\epsilon  \, \cup \,  C_N^+ \, \cup \,  C_N$, where $C_N$ is the arc of radius $N + 1/2$ on the left half-plane illustrated in 
Figure~\ref{fig.2}. For $s \in C_N$ we have $\cos\theta < 0$ and for 
the leading term in eq.\eqref{exp.log.integrand} we have  
$$
\alpha \cos\theta R \log{R} < 0 , 
$$
which implies that  
$$
\lim_{N\to \infty} \int_{C_N} (\kappa t)^{-s} 
\frac{\Gamma(s)G(2+\nu-\alpha-s)}{G(1+\nu-s)} ds = 0 . 
$$
Thus, from the residue theorem, 
\begin{equation}
\begin{aligned}
f(t) & = f_0 \sum_{n=0}^\infty \underset{s=-n}{\operatorname{Res}} \left[ (\kappa t)^{-s} \frac{\Gamma(s) G(2+\nu-\alpha-s)}{  G(1+\nu-s)} \right] \\[1ex]
& = f_0 \sum_{n=0}^\infty   \frac{G(2+n-\alpha+\nu)}{G(1+n+\nu)} \frac{(-\kappa t)^n}{n!} 
\end{aligned}
\end{equation}
Formula \eqref{eq.solution} is obtained by choosing $f_0$ in such a way that the initial condition $f(0) = 1$ is satisfied,  that is, 
\begin{equation}
\label{def.f_0}
f_0 = \frac{G(1+\nu)}{G(2-\alpha+\nu)}. 
\end{equation}

\end{proof}

\subsection{The $\boldsymbol{G}$-hypergeometric functions}

The function $\Psi_\alpha^\nu(-\kappa t)$ in eq.\eqref{eq.solution} can be
written in a more convenient way after defining the $G$-Pochhammer symbol. 

\begin{definition}
Let $G(\cdot)$ be the Barnes $G$-function (see Appendix~\ref{appendix.Barnes}). 
We define the $G$-Pochhammer symbol as 
\begin{equation}
\label{poch-like}
\Lbrack a \Rbrack_n = \frac{G(a + n)}{G(a)} ,  
\end{equation}
where $a \in \mathbb{C}$.
\end{definition}

Note that the ordinary Pochhammer symbol $(a)_n$ can be expressed as a particular case of the $G$-Pochhammer symbol as 
\begin{equation}
\label{poch.2.G.poch}
    (a)_n = \frac{\Lbrack a+1 \Rbrack_n}{\Lbrack a \Rbrack_n} .
\end{equation}

The next step is to define the $G$-hypergeometric functions of type $(p,q)$. 

\begin{definition} 
Let $p,q \in \mathbb{N}$. 
The $G$-hypergeometric function of the type $(p,q)$ is defined as 
\begin{equation} 
\sideset{_{p}^{}}{_{q}^{}}{\operatorname{\mathfrak{F}}}(a_1,\ldots,a_p;b_1,\ldots,b_q;z) = 
\sum_{n=0}^\infty \frac{\Lbrack a_1\Rbrack_n \cdots \Lbrack a_p\Rbrack_n}{ \Lbrack b_1\Rbrack_n \cdots \Lbrack b_q \Rbrack_n} \frac{z^n}{n!} , 
\end{equation}
where $\Lbrack \cdot \Rbrack_n$ is the $G$-Pochhammer symbol defined in eq.\eqref{poch-like}. 
\end{definition}

Using eq.\eqref{poch.2.G.poch} we can write the ordinary hypergeometric function of type $(p,q)$ in terms of the $G$-hypergeometric functions as 
\begin{equation}
\begin{split}
& \sideset{_p^{}}{_q^{}}{\operatorname{\mathit{F}}}(a_1,\ldots,a_p;b_1,\ldots,b_q;z) \\
& \qquad \qquad = 
\sideset{_{p+q}^{}}{_{p+q}^{}}{\operatorname{\mathfrak{F}}}(a_1+1,\ldots,a_p+1,b_1,\ldots,b_q;a_1,\ldots,a_q,b_1+1,\ldots,b_q+1;z) . 
\end{split}
\end{equation}
We also note that $G$-hypergeometric functions can be seen as particular cases
of Fox-Barnes $J$-function \cite{Vaz-Fox-Barnes}. 

With these definitions in place, we can write $\Psi_\alpha^\nu(-\kappa t)$ as 
a particular example of a $G$-hypergeometric function, namely 
\begin{equation}
    \label{eq.solution.final}
    \Psi_\alpha^\nu(-\kappa t) =  
\sideset{_1^{}}{_1^{}}{\operatorname{\mathfrak{F}}}(2+\nu-\alpha;1+\nu;-\kappa t) .
\end{equation}

\begin{lem} 
Let $\sideset{_1^{}}{_1^{}}{\operatorname{\mathfrak{F}}}(a;b;z)$ be the
$G$-hypergeometric function of type $(1,1)$, 
\begin{equation}
\label{general.chg}
 \sideset{_1^{}}{_1^{}}{\operatorname{\mathfrak{F}}}(a;b;z) = 
\sum_{n=0}^\infty \frac{\Lbrack a \Rbrack_n}{\Lbrack b \Rbrack_n} \frac{z^n}{n!} . 
\end{equation} 
If $a-1 < b$ this series converges for all $z \in \mathbb{R}$. 
\end{lem}

\begin{proof}
The converge of the series in eq.\eqref{general.chg} can be analyzed
using the ratio test. We have
\begin{equation}
\begin{aligned}
L   = & \lim_{n\to \infty}\left|\frac{z}{n+1}  \frac{G(a+n+1) G(b+n)}{G(a+n) G(b+n+1)}\right| = |z| \lim_{n\to \infty}\left|\frac{1}{n+1}  \frac{\Gamma(a+n)}{\Gamma(b+n)}\right| \\[1ex]
= & |z| \lim_{n\to \infty}\left|\frac{a-1+n}{n+1}\right| \, \left| \frac{\Gamma(a-1+n)}{\Gamma(b+n)}\right| = |z| \lim_{n\to \infty} \left| \frac{\Gamma(a-1+n)}{\Gamma(b+n)}\right|
\end{aligned}
\end{equation}
Now we recall Gautschi's inequality \cite{Gautschi}: 
for $x > 0$ and $\sigma\in (0,1)$, it holds 
$$
x^{1-\sigma} < \frac{\Gamma(x+1)}{\Gamma(x+\sigma)} < (x+1)^{1-\sigma} ,
$$
or 
$$
\frac{1}{(x+1)^{1-\sigma}} < \frac{\Gamma(x+\sigma)}{\Gamma(x+1)} < \frac{1}{x^{1-\sigma} } .
$$
Thus, provided $\sigma = a - b <1$, we can write, using Gautschi's inequality, that 
$$
 \frac{\Gamma(a-1+n)}{\Gamma(b+n)} = 
 \frac{\Gamma((b+n-1)+(a-b))}{\Gamma((b+n-1)+1)} < \frac{1}{(b+n-1)^{(1-(a-b))}}
$$
and then 
$$
 \lim_{n\to \infty} \left| \frac{\Gamma(a-1+n)}{\Gamma(b+n)}\right| = 0 \quad \text{for}\;  
 a - 1 < b  . 
$$
Thus, if $a - 1 < b$, the series in eq.\eqref{general.chg} converges for 
all $z \in \mathbb{R}$. 
\end{proof}

\bigskip

For the function $\sideset{_1^{}}{_1^{}}{\operatorname{\mathfrak{F}}}(2+\nu-\alpha;1+\nu;-\kappa t)$ we have $a=2+\nu-\alpha$ and $b=1+\nu$, and since $\alpha > 0$, we have 
$$
a - 1 = 1 - \alpha + \nu < 1 + \nu = b .
$$
Thus we have: 

\begin{coro}
The function $\Psi_\alpha^\nu(-\kappa t) =  
\sideset{_1^{}}{_1^{}}{\operatorname{\mathfrak{F}}}(2+\nu-\alpha;1+\nu;-\kappa t)$ defined by the series 
$$
\sideset{_1^{}}{_1^{}}{\operatorname{\mathfrak{F}}}(2+\nu-\alpha;1+\nu;-\kappa t) = \sum_{n=0}^\infty \frac{\Lbrack 2+\nu-\alpha\Rbrack_n}{\Lbrack 1+\nu\Rbrack_n} \frac{(-\kappa t)^n}{n!} 
$$
converges for all $t \in \mathbb{R}$. 
\end{coro}

\subsection{Complete monotonicity of the solution} 

We now prove a key property of the solution to the first-order differential equation, which is represented by its complete monotonicity.
\begin{theorem}\label{second}
For every $\alpha \in (0,1)$ and every $\nu > -1,$ the function $f(t) = \Psi_\alpha^\nu(-\kappa t)$, given by  eq.\eqref{eq.solution} or by eq.\eqref{inverse.mellin}, is completely monotone. 

\end{theorem}
\begin{proof}
Let $\{\sigma^{(\alpha)}_u, \, u\ge 0\}$ stand for the standard $\alpha-$stable subordinator, i.e. such that 
\begin{equation}{\mathbb E} [e^{-\lambda \sigma^{(\alpha)}_u}] \, = \,e^{-u \lambda^\alpha} ,\qquad u, \lambda \geq 0.\label{subord}\end{equation}
Let us define the random variable
$$X_\alpha \, =\, \int_0^\infty \left(\frac{1}{1 + \sigma_u^{(\alpha)}}\right)^{1+\alpha} du.$$ 
Then, it follows from Proposition 2.4 and Corollary 1.5 in \cite{HSF} that
$${\mathbb E} (X_\alpha^s) \, =\, \frac{G(2+s) G(1+\alpha)}{G(1+\alpha +s) G(2)}$$
for every $s > -(1+\alpha).$ Therefore, we have
\begin{eqnarray*}
{\mathbb E} (e^{-t (\kappa X_\alpha)}) & = & \sum_{n\ge 0} \frac{(-\kappa t)^n {\mathbb E}(X_\alpha^n)}{n!} \\
& = &
\frac{G(1+\alpha)}{G(2)}\sum_{n=0}^\infty   \frac{G(2+n)}{G(1+n+\alpha)} \frac{(-\kappa t)^n}{n!}\, =\, \Psi_\alpha^\alpha(-\kappa t)
\end{eqnarray*}
and this shows the result for every $\alpha \in (0,1)$ and $\nu = \alpha.$ Suppose next $\alpha\neq\nu > -1$ and set $X_{\alpha, \nu} = X_\alpha^{(\nu-\alpha)},$ with the usual notation for the size-bias $X^{(t)}$ of a positive random variable $X$, that is
$${\mathbb E}[f(X^{(t)})] = \frac{{\mathbb E}[X^t f(X)]}{{\mathbb E} [X^t]}$$
for every $f$ bounded continuous and every $t\in {\mathbb R}$ such that $ {\mathbb E} (X^t) < \infty.$ Then, with the notations of \cite{HSF}, we have
$$X_{\alpha, \nu} \, \sim\, \frac{1}{\Gamma(1+\alpha)} \,T (1+\alpha, 1, 1-\alpha)^{(\nu -\alpha)}\, \sim\, \frac{\Gamma(\nu-\alpha +2)}{\Gamma(1+\nu)} \, T (1+\nu, 1, 1-\alpha).$$ 
 Applying again Proposition 2.4 in \cite{HSF} leads to
\begin{equation}\label{mom}{\mathbb E} (X_{\alpha,\nu}^s) \, =\, \frac{G(2+s+\nu -\alpha) G(1+\nu)}{G(1+\nu +s) G(2+\nu-\alpha)}\end{equation}
for every $s > -(1 +\nu),$ and then to
\begin{equation}\label{LTpsi}
{\mathbb E} e^{-t (\kappa X_{\alpha,\nu})}  =\frac{G(1+\nu)}{G(2-\alpha+\nu)}\sum_{n=0}^\infty   \frac{G(2+n-\alpha+\nu)}{G(1+n+\nu)} \frac{(-\kappa t)^n}{n!}= \Psi_\alpha^\nu(-\kappa t), 
\end{equation}
which shows the required property.
\end{proof}

\begin{remark} In the case $\nu =\alpha,$ we have the stochastic representation
$$X_{\alpha, \alpha} \, =\, \int_0^\infty \left(\frac{1}{1 + \sigma_u^{(\alpha)}}\right)^{1+\alpha} du$$ 
as the Riemannian integral of an $\alpha-$stable subordinator. It is worth mentioning that there is a stochastic representation with another stable process in the reguralized Caputo case $\nu = 0.$ If we consider the spectrally positive $(2-\alpha)-$stable process $\{L_u^{(\alpha)}, \, u\ge 0\}$ starting from one, with normalization
$${\mathbb E} [e^{-\lambda L_u^{(\alpha)}}] \, = \,e^{u \lambda^{2-\alpha} - \lambda} ,\qquad u, \lambda \geq 0,$$
it follows namely from \eqref{mom} and a combination of Theorem 1.2 (a) and Proposition 2.4. in \cite{HSF} that 
$$X_{\alpha, 0} \, \sim\, \left(\int_0^T \Big(L_u^{(\alpha)}\Big)^{\alpha -1} du\right)^{-1}$$
with the notation $T = \inf\{ u \ge 0, L_u = 0\}.$  
\end{remark}

\subsection{Asymptotic expansion of the solution} 

In this section we want to show the following. 

\begin{theorem}
Let $\alpha \in (0,1)$, $\nu \in \mathbb{R}$ and $m \in \mathbb{N}$. Then, the function $\Psi_\alpha^\nu(-\kappa t)$ admits the asymptotic expansion
\begin{equation}
\label{asymptotic.psi}
\Psi_\alpha^\nu(-\kappa t) =  -\frac{G(1+\nu)}{G(2+\nu-\alpha)}
\sum_{k=1}^m \rho_k  
(\kappa t)^{-(\nu+k)} 
 + \mathcal{O}\left(|\kappa t|^{-(\nu+m+1/2)}\right),
\end{equation}
where 
\begin{equation}
\label{residue.1}
\rho_k = \underset{\varepsilon=0}{\operatorname{Res}} \left[\frac{(\kappa t)^{-\varepsilon} \Gamma(\nu+k+\varepsilon)G(2-\alpha-k-\varepsilon)}{  G(1-k-\varepsilon)} \right], \qquad (k = 1,2,\ldots,m).
\end{equation}
\end{theorem}

\begin{proof}
We will look for an asymptotic expansion for $\Psi_\alpha^\nu(-\kappa t)$. 
To this aim, we cannot use an arc on the right half-plane (as we did 
on the left-plane illustrated in Figure~\ref{fig.2}), because for this 
arc, we would have $\cos\theta > 0$ (since $|\theta| < \pi/2$). As a consequence, 
the leading term in the expansion 
in eq.\eqref{exp.log.integrand} would be    
$$
\alpha \cos\theta R \log{R} > 0 ,
$$
and thus the integral along this arc would diverge for $R \to \infty$. 
As an alternative, we will consider the closed contour $C = C_0 \, \cup \, C_\uparrow \, \cup \, C^\prime \, \cup C_\downarrow$ illustrated in Figure~\ref{fig.3}. 

The integral along $C_0$ for $\tau \to \infty$ defines the 
function $\Psi_\alpha^\nu(-\kappa t)$. Thus, using the residue theorem, 
we have 
\begin{equation}
\begin{aligned}
\Psi_\alpha^\nu(-\kappa t) = & -
2\pi i \sum_{k=1}^m \underset{s = \nu + k}{\operatorname{Res}} \left[\frac{1}{2\pi i}\frac{G(1+\nu)}{G(2+\nu-\alpha)} (\kappa t)^{-s} \frac{\Gamma(s) G(2+\nu-\alpha-s)}{  G(1+\nu-s)} \right]\\[1ex]
& - \frac{1}{2\pi i}\frac{G(1+\nu)}{G(2+\nu-\alpha)}  \lim_{\tau\to \infty} \left(\int_{C_\uparrow} + \int_{C^\prime} + \int_{C_\downarrow}\right)
(\kappa t)^{-s} \frac{\Gamma(s) G(2+\nu-\alpha-s)}{  G(1+\nu-s)}\, ds
\end{aligned}
\end{equation}

In the appendices we show that 
$$
\frac{1}{2\pi i}\frac{G(1+\nu)}{G(2+\nu-\alpha)}  \lim_{\tau\to \infty}  \int_{C_\uparrow/C_\downarrow}  
(\kappa t)^{-s} \frac{\Gamma(s) G(2+\nu-\alpha-s)}{  G(1+\nu-s)}\, ds = 0 
$$
and 
$$
|J_{\nu,m}(\alpha)| \leq  \mathcal{I}(N) \Theta(\nu,m,\alpha)  |\kappa t|^{-(\nu+m+1/2)} ,
$$
where 
$$
 J_{\nu,m}(\alpha) =   -\frac{1}{2\pi i}\frac{G(1+\nu)}{G(2+\nu-\alpha)}
\int_{(\nu+m+1/2)-i\infty}^{(\nu+m+1/2)+i\infty}(\kappa t)^{-s}  \Gamma(s) \frac{G(2+\nu-\alpha-s)}{G(1+\nu-s)} \, ds .
$$
Thus we have that 
\begin{equation}
\Psi_\alpha^\nu(-\kappa t) = - \frac{G(1+\nu)}{G(2+\nu-\alpha)}
\sum_{k=1}^m \underset{s = \nu + k}{\operatorname{Res}} \left[ (\kappa t)^{-s} \frac{\Gamma(s) G(2+\nu-\alpha-s)}{  G(1+\nu-s)} \right] + \mathcal{O}\left(|\kappa t|^{-(\nu+m+1/2)}\right) , 
\end{equation}
or, defining $\varepsilon = s-\nu-k$,    
\begin{equation}
\Psi_\alpha^\nu(-\kappa t) =  -\frac{G(1+\nu)}{G(2+\nu-\alpha)}
\sum_{k=1}^m \rho_k  
(\kappa t)^{-(\nu+k)} 
 + \mathcal{O}\left(|\kappa t|^{-(\nu+m+1/2)}\right) , 
\end{equation}
where 
\begin{equation}
\rho_k = \underset{\varepsilon=0}{\operatorname{Res}} \left[\frac{(\kappa t)^{-\varepsilon} \Gamma(\nu+k+\varepsilon)G(2-\alpha-k-\varepsilon)}{  G(1-k-\varepsilon)} \right] \qquad (k = 1,2,\ldots,m) .
\end{equation}
\end{proof}

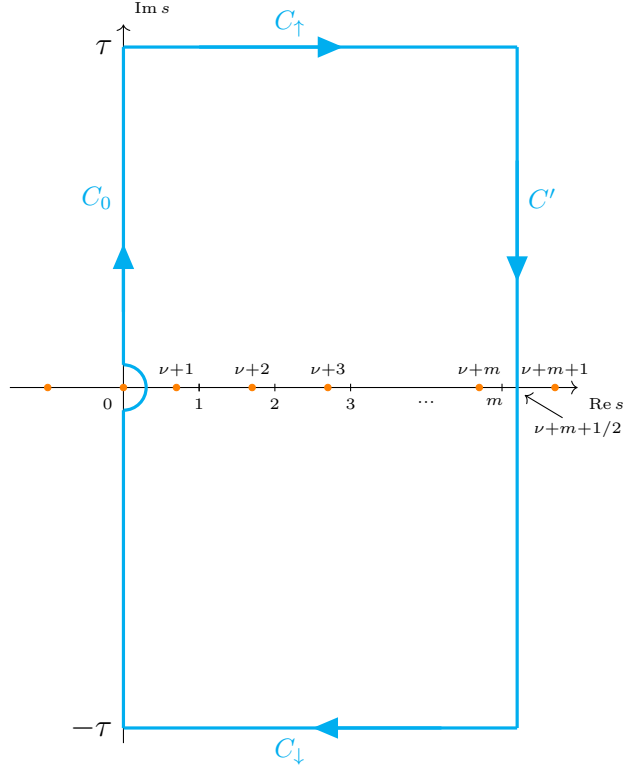
\begin{figure}[h] 
\begin{center}
\begin{tikzpicture}[scale=1]
\draw[->] (-1.5,0) -- (6,0) node[below right] {${\scriptscriptstyle \operatorname{Re}s}$};
\draw[->] (0,-4.7) -- (0,4.8) node[above right] {${\scriptscriptstyle \operatorname{Im}s}$};
\draw[orange,fill=orange] (0,0) circle (1.2pt);
\draw[orange,fill=orange] (-1,0) circle (1.2pt);
\draw[-] (2,0.05) -- (2,-0.05);;
\draw[-] (3,0.05) -- (3,-0.05);
\draw[-] (5,0.05) -- (5,-0.05);;
\draw[orange,fill=orange] (1.7,0) circle (1.2pt);
\draw[-] (1,0.05) -- (1,-0.05);
\draw[orange,fill=orange] (0.7,0) circle (1.2pt);
\draw[orange,fill=orange] (2.7,0) circle (1.2pt);
\draw[orange,fill=orange] (4.7,0) circle (1.2pt);
\draw[orange,fill=orange] (5.7,0) circle (1.2pt);
\draw[-] (1,0.05) -- (1,-0.05);
\node[below] at (1,0) {${\scriptscriptstyle 1}$};
\node[below] at (2,0) {${\scriptscriptstyle 2}$};
\node[below] at (3,0) {${\scriptscriptstyle 3}$};
\node[below] at (4,0) {${\scriptscriptstyle \cdots}$};
\node[below] at (4.9,0) {${\scriptscriptstyle m}$};
\node[below left] at (0,0) {${\scriptscriptstyle 0}$};
\draw[cyan,very thick] (0,0.3) -- (0,4.5);
\draw[cyan,very thick] (5.2,-4.5) -- (5.2,4.5);
\draw[cyan,very thick] (0,-0.3) -- (0,-4.5);
\draw[cyan,very thick] (0,-4.5) -- (5.2,-4.5);
\draw[cyan,very thick] (0,4.5) -- (5.2,4.5); 
\draw[->,>=triangle 45,cyan,very thick] (0,1) -- (0,1.9);
\draw[->,>=triangle 45,cyan,very thick] (5.2,3) -- (5.2,1.4);
\draw[->,>=triangle 45,cyan,very thick] (1,4.5) -- (2.9,4.5);
\draw[->,>=triangle 45,cyan,very thick] (4.2,-4.5) -- (2.5,-4.5);
\draw[cyan, very thick] (0,-0.3) arc (-90:90:0.3);
\node[left] at (0,2.5) {\footnotesize ${\textcolor{cyan}{C_0}}$};
\node[right] at (5.2,2.5) {\footnotesize ${\textcolor{cyan}{C^\prime}}$};
\node[above] at (2.2,4.5) {\footnotesize $\textcolor{cyan}{C_\uparrow}$};
\node[below] at (2.2,-4.5) {\footnotesize $\textcolor{cyan}{C_\downarrow}$};
\node[above] at (0.7,0) {$\scriptscriptstyle \nu+1$};
\node[above] at (1.7,0) {$\scriptscriptstyle \nu+2$};
\node[above] at (2.7,0) {$\scriptscriptstyle \nu+3$};
\node[above] at (4.7,0) {$\scriptscriptstyle \nu+m$};
\node[above] at (5.7,0) {$\scriptscriptstyle \nu+m+1$};
\draw[->,thin] (5.8,-0.4) -- (5.3,-0.1);
\node[below] at (6,-0.3) {$\scriptscriptstyle \nu+m+1/2$};
\node[left] at (0,-4.5) {$\footnotesize -\tau$};
\node[left] at (0,4.5) {$\footnotesize \tau$};
\end{tikzpicture}
\end{center}
\caption{Integration contour for eq.\eqref{inverse.mellin}.\label{fig.3}}
\end{figure}

Since the zeros of $G(s)$ at $s=-m$ ($m=0,1,2,\ldots$) are zeros of order $m+1$, it follows that the pole of $1/G(1-k-\varepsilon)$ at $\varepsilon=0$ is a 
simple pole for $k=1$, a pole of order 2 for $k=2$, etc. In Appendix~\ref{residues} 
we evaluate $\rho_1$ and $\rho_2$ and the result is 
\begin{equation}
\label{res.s1}
\rho_1 = -\Gamma(\nu+1) G(1-\alpha) 
\end{equation}
and 
\begin{equation}
\label{res.s2}
\rho_2 = \Gamma(\nu+2) G(-\alpha) [\log{\kappa t} - \psi(\nu+2) + \alpha -(\alpha+1)\psi(-\alpha) - 2\gamma + 2] ,
\end{equation}
where $\gamma$ is the Euler-Mascheroni constant and $\psi(\cdot)$ is the digamma function.   

So using eq.\eqref{res.s1} and eq.\eqref{res.s2} in eq.\eqref{asymptotic.psi}, we obtain  
\begin{equation}\label{asypsi}
\begin{aligned}
\Psi_\alpha^\nu(-\kappa t) = & 
\frac{G(1+\nu)(\kappa t)^{-\nu}}{G(2+\nu-\alpha)} \bigg[ \Gamma(\nu+1)G(1-\alpha) 
(\kappa t)^{-1} \\[1ex]
& -\Gamma(\nu+2)G(-\alpha) (\kappa t)^{-2}\log{\kappa t} \\[1ex]
& + \Gamma(\nu+2)G(-\alpha)\left[\psi(\nu+2)-\alpha + (\alpha+1)\psi(-\alpha)+2\gamma-2\right] (\kappa t)^{-2} \\[1ex]
&+ 
\mathcal{O}(|\kappa t|^{-5/2})\bigg] . 
\end{aligned}
\end{equation}

\subsection{Laplace transform of the solution}


Let us start with the following.

\begin{theorem}
    Let $\mathscr{L}_t[f(t);z]$ be the Laplace transform of the function $f(t)$ with conjugate variable $z$. Here $t$ denotes the time variable and $z$ is the Laplace dual variable. Then 
\begin{equation}
\label{laplace.psi}
\begin{split} 
& \mathscr{L}_t[t^\mu \Psi_\alpha^\nu(-\kappa t);z] \\[1ex]
& \quad = 
\frac{1}{z^{1+\mu}} \frac{G(1+\nu)}{G(2+\nu-\alpha)} \frac{1}{2\pi i} 
\int_{c-\infty}^{c+\infty} \left(\frac{\kappa}{z}\right)^{-s} \Gamma(1+\mu-s) 
\Gamma(s) \frac{G(2+\nu-\alpha-s)}{G(1+\nu-s)} \, ds 
\end{split}
\end{equation}
with $\mu > -1$. 
\end{theorem}

\begin{proof} 
We saw in Theorem \ref{main} that 
\begin{equation}
\Psi_\alpha^\nu(-\kappa t) = 
\frac{G(1+\nu)}{G(2+\nu-\alpha)} \frac{1}{2\pi i} 
\int_{c-\infty}^{c+\infty} (\kappa t)^{-s} 
\Gamma(s) \frac{G(2+\nu-\alpha-s)}{G(1+\nu-s)} \, ds 
\end{equation}
we have $0 < c< 1+\nu $ with $\nu > -1$, and that this integral converges absolutely for $|\operatorname{arg}(\kappa t)| < (1-\alpha)\pi/2$. Taking the Laplace transform of $f(t) = t^\mu \Psi_\alpha^\nu(-\kappa t)$ and assuming we can change the order of integration gives eq.\eqref{laplace.psi}. 

Now let us see whether it is well defined. In the integration giving $\Gamma(1+\mu-s)$ it is required that $\mu-\operatorname{Re}s  > -1$. Since $0 < c < 1+ \nu$, if we assume that $\mu > -1$ we can choose $c$ such that 
$0 < c < \min(1+\nu,1+\mu)$, and so the condition $\mu - \operatorname{Re}s> -1$ is satisfied. 

In relation to the convergence of the integral, let us analyze 
\begin{equation}
\log\bigg| \left(\frac{\kappa}{z}\right)^{-s} \Gamma(1+\mu-s) 
\Gamma(s) \frac{G(2+\nu-\alpha-s)}{G(1+\nu-s)}\bigg| = 
\log|\Gamma(1+\mu-s)| + 
\log|J(z^{-1},s)| ,
\end{equation}
with $\log|J(z^{-1},s)|$ as in eq.\eqref{log.integrand}. Writing 
$s = R\operatorname{e}^{i\theta}$, the Stirling formula gives 
\begin{equation}
\log|\Gamma(1+\mu-s)| = -\cos\theta R\log{R} + (\theta\sin\theta + \cos\theta - \pi\sin|\theta|)R + (\mu+1/2)\log{R} + \mathcal{O}(1) , 
\end{equation}
while for $\log{|J(z^{-1},s)|}$ we have eq.\eqref{exp.log.integrand}. So we obtain 
\begin{equation}
\label{log.integrand.bis}
\begin{split}
& \log\bigg| \left(\frac{\kappa}{z}\right)^{-s} \Gamma(1+\mu-s) 
\Gamma(s) \frac{G(2+\nu-\alpha-s)}{G(1+\nu-s)}\bigg| \\[1ex]
& = (\alpha-1)\cos\theta R\log{R} + A^\prime R + 
\left(\frac{(\alpha+\nu)(\alpha+\nu-2)}{2} + \mu + \frac{1}{2}\right)\log{R} + 
\mathcal{O}(1) ,
\end{split}
\end{equation}
where 
\begin{equation}
A^\prime = (1-\alpha)(\theta\sin\theta+\cos\theta) + (2-\alpha)\pi \sin|\theta| 
-\cos\theta \log\left|\frac{\kappa}{z}\right| + \sin\theta \operatorname{arg}\left(\frac{\kappa}{z}\right) .
\end{equation}
Thus repeating the same argument in Theorem \ref{main} and for $\kappa > 0$, we conclude 
that the integral along $(c-i\infty,c+i\infty)$ converges absolutely for 
\begin{equation}
    |\operatorname{arg}z| < (3-\alpha)\frac{\pi}{2} .
\end{equation}
\end{proof}

Eq.\eqref{log.integrand.bis} guarantees that for $\alpha < 1$ we can consider a 
closed contour by adding to the contour $(c-i\infty,c+i\infty)$ an arc on the right half-plane and then 
use the residue theorem to obtain a series representation for 
$\mathscr{L}_t[t^\mu \Psi_\alpha^\nu(-\kappa t);z]$. 
The poles to the right of the contour $(c-i\infty,c+i\infty)$ are those arising from $\Gamma(1+\mu-s)$ in the numerator (that is, $s_n^{(1)} = 1 + \mu + n$ for $n = 0,1,2,\ldots$) and from $G(1+\nu-s)$ in the denominator (that is, $s_m^{(2)} = 1 + \nu + m$ for $m = 0,1,2,\ldots$) of the integrand. 
The poles $s_n^{(1)}$ are simple poles, whereas the poles $s_m^{(2)}$ are poles of order $m+1$. Therefore, finding the general term of the series representation of $\mathscr{L}_t[t^\mu \Psi_\alpha^\nu(-\kappa t);z]$ is a difficult task due to the poles $s_m^{(2)}$. However, the first term of the series can be obtained explicitly, and we proceed by distinguishing the cases according to the relation between $\nu$ and $\mu$.

\begin{theorem}
\label{theorem.laplace.3cases}
For the Laplace transform $\mathscr{L}_t[t^\mu \Psi_\alpha^\nu(-\kappa t);z]$ 
it holds, as $z \to 0$,
\begin{equation}
\label{laplace.first.term}
\begin{split}
& \mathscr{L}_t[t^\mu \Psi_\alpha^\nu(-\kappa t);z] \\[1ex]
& = 
\begin{cases}
\displaystyle \frac{\Gamma(1+\mu)G(1+\nu)G(1+\nu-\mu-\alpha)}{\kappa^{1+\mu} G(2+\nu-\alpha) G(\nu-\mu)} 
+ \mathcal{O}\big(z^{\min(\nu-1,1)}\big) , \quad \text{if} \; \nu > \mu \\[2ex]
\displaystyle \frac{G(2+\nu)\Gamma(\mu-\nu)G(1-\alpha)}{\kappa^{1+\nu} G(2+\nu-\alpha)}\, 
\frac{1}{z^{\mu-\nu}} + \mathcal{O}\big(z^{\min(0,1+\nu-\mu)}\big) , \quad 
\text{if}\; \mu > \nu \\[2ex]
\displaystyle -\frac{G(2+\nu)G(1-\alpha)}{\kappa^{1+\nu}G(2+\nu-\alpha)}\log z + 
\mathcal{O}(1) , \quad \text{if} \; \nu = \mu 
\end{cases}
\end{split}
\end{equation}
\end{theorem}

\begin{proof}
Let us start with the case (i) $\nu > \mu$. In this case the first pole is $s_0^{(1)}$ and the second one is $\min(s_1^{(1)},s_0^{(2)})$. So we can write
\begin{equation}
 \mathscr{L}_t[t^\mu\Psi_\alpha^\nu(-\kappa t);z] =
-\frac{G(1+\nu)}{G(2+\nu-\alpha)} \frac{1}{z^{1+\mu}} \rho_1^{\text{(i)}} + \frac{1}{z^{1+\mu}}\mathcal{O}\big(z^{\min(2+\mu,1+\nu)}\big) ,
\end{equation}
where 
\begin{equation}
\label{residue.case.i}
 \rho_1^{\text{(i)}}  = \underset{s=1+\mu}{\operatorname{Res}} \left[ \left(\frac{\kappa}{z}\right)^{-s} \Gamma(1+\mu-s)\Gamma(s)\frac{G(2+\nu-\alpha-s)}{G(1+\nu-s)} \right]
\end{equation}
and the minus sign is due to the orientation of the closed contour. 
In Appendix~\ref{residues} we show that 
\begin{equation}
    \rho_1^{\text{(i)}} =  -\left(\frac{\kappa}{z}\right)^{-(1+\mu)} \Gamma(1+\mu) \frac{G(1+\nu-\mu-\alpha)}{G(\nu-\mu)} 
\end{equation}
which in the above expression gives the case $\nu > \mu$ in eq.\eqref{laplace.first.term}. 

Next we consider the case (ii) $\mu > \nu$. Now the first pole is $s_0^{(2)}$ and the second one is 
$\min(s_1^{(2)},s_0^{(1)})$. So we have 
\begin{equation}
\mathscr{L}_t[t^\mu\Psi_\alpha^\nu(-\kappa t);z] =
-\frac{G(1+\nu)}{G(2+\nu-\alpha)} \frac{1}{z^{1+\mu}} \rho_1^{\text{(ii)}} + \frac{1}{z^{1+\mu}}\mathcal{O}\big(z^{\min(2+\nu,1+\mu)}\big) ,
\end{equation}
where 
\begin{equation}
\label{residue.case.ii}
    \rho_1^{\text{(ii)}} =  \underset{s=1+\nu}{\operatorname{Res}} \left[ \left(\frac{\kappa}{z}\right)^{-s} \Gamma(1+\mu-s)\Gamma(s)\frac{G(2+\nu-\alpha-s)}{G(1+\nu-s)} \right] .
\end{equation}
In Appendix~\ref{residues} we show that 
\begin{equation}
\rho_1^{\text{(ii)}} = -\left(\frac{\kappa}{z}\right)^{-(1+\nu)} \Gamma(\mu-\nu)\Gamma(1+\nu) G(1-\alpha)
\end{equation}
which gives the case $\mu > \nu$ in eq.\eqref{laplace.first.term}. 

Finally, for the case (iii) $\nu = \mu$ the first pole is $s_0^{(1)} = s_0^{(2)} = 1 + \nu$, the second one is $s_1^{(1)} = s_1^{(2)} = 2 + \nu$, etc. Note that $s_0^{(1)} = s_0^{(2)}$ is a pole of order two, $s_1^{(1)} = s_1^{(2)}$ is a pole of order three, etc. We have 
\begin{equation}
\mathscr{L}_t[t^\mu \Psi_\alpha^\nu(-\kappa t);z] = 
-\frac{G(1+\nu)}{G(2+\nu-\alpha)} \frac{1}{z^{1+\nu}} \rho_1^{\text{(iii)}} + 
\frac{1}{z^{1+\nu}} \mathcal{O}(z^{2+\nu}) 
\end{equation}
where 
\begin{equation}
\label{residues.case.iii}
    \rho_1^{\text{(iii)}} = \underset{s=1+\nu}{\operatorname{Res}} 
    \left[ \left(\frac{\kappa}{z}\right)^{-s}\Gamma(1+\nu-s) \Gamma(s)\frac{G(2+\nu-\alpha-s)}{G(1+\nu-s)} \right]
\end{equation}
In Appendix~\ref{residues} we show that 
\begin{equation}
\rho_1^{\text{(iii)}} = \left(\frac{\kappa}{z}\right)^{-(1+\nu)} \Gamma(1+\nu) G(1-\alpha) 
\left[2\gamma + \psi(1+\nu) -\alpha + \alpha \psi(1-\alpha) + \log\left(\frac{z}{\kappa}\right)\right]
\end{equation}
which gives the case $\mu = \nu$ in eq.\eqref{laplace.first.term}. 

\end{proof}

The three regimes in Theorem~\ref{theorem.laplace.3cases} reflect the competition between the parameters $\mu$ and $\nu$, which determine the dominant contribution of the Mellin--Barnes integrand. When $\nu > \mu$, the leading behavior is governed by the simple poles of $\Gamma(1+\mu-s)$, yielding a finite limit as $z \to 0$. In contrast, when $\mu > \nu$, the dominant contribution arises from the poles of $G(1+\nu-s)$, producing a power-law divergence in $z$. The critical case $\nu = \mu$ corresponds to a coalescence of poles, leading to a logarithmic term.

 \section{Second-order differential equation}

Let us consider the second order homogeneous fractional differential equation 
\begin{equation}
\label{2nd.order.eq}
\left(\mathscr{D}_t^{(\alpha,\nu)}\right)^2 f(t) + a \mathscr{D}_t^{(\alpha,\nu)}f(t) + b f(t) = 0 . 
\end{equation}
where $\left(\mathscr{D}_t^{(\alpha,\nu)}\right)^2 f(t) = 
\mathscr{D}_t^{(\alpha,\nu)}\left(\left(\mathscr{D}_t^{(\alpha,\nu)}\right) f(t)\right) $ is the sequential beta-weighted derivative of order two, 
and $a$ and $b$ are real constants. 

The use of sequential fractional 
derivatives is motivated by the fact that, 
in general, fractional differential operators do not satisfy the semi-group property. 
In particular, for $\mathscr{D}_t^{(\alpha,\nu)}$, in general we have  
$\mathscr{D}_t^{(\alpha,\nu)}\left(\mathscr{D}_t^{(\alpha,\nu)} f
\right) \neq \mathscr{D}_t^{(2\alpha,\nu)} f$. Sequential derivatives therefore 
provide a natural generalization of higher-order differential equations obtained by 
iterating the same dynamical mechanism. This point of view is also motivated by the
structure of several classical equations of mathematical physics, which arise from
the successive application or combination of first-order laws. 
For example, the diffusion equation follows from the continuity equation together 
with Fick's law, while Newton's second law may be written as $F=\frac{dp}{dt}$ with
$p=m\frac{dx}{dt}$, leading to a second-order equation for the position. 
This formulation is particularly relevant in systems with variable mass, 
where the momentum balance equation $F=\frac{d}{dt}(mv)$ cannot be reduced 
to the simple form $F=m\,d^2x/dt^2$. From this perspective, 
sequential fractional equations provide a framework for extending
such constructions to systems with memory, 
where the same fractional dynamical mechanism acts recursively 
or in multiple stages.  

\begin{theorem}
Let $\lambda_1$ and $\lambda_2$ be such that
$$
a=\lambda_1+\lambda_2,
\qquad 
b=\lambda_1\lambda_2 ,
$$
and consider the differential equation \eqref{2nd.order.eq}. 
Then:
\begin{enumerate}
\item[\textrm{(i)}] If $\lambda_1\neq \lambda_2$, the general solution of eq.\eqref{2nd.order.eq} is
\begin{equation}
\label{second.order.sol.1}
f(t)=c_1\Psi_\alpha^\nu(-\lambda_1 t)
+c_2\Psi_\alpha^\nu(-\lambda_2 t).
\end{equation}
where $\Psi_\alpha^\nu$ is defined in eq.\eqref{eq.solution.final} 
and the constants $c_1$ and $c_2$ are chosen so that $f(t)$ is real-valued
whenever $\lambda_1$ and $\lambda_2$ are complex conjugates. 
\item[\textrm{(ii)}] If $\lambda_1=\lambda_2$, the general solution of eq.\eqref{2nd.order.eq} is
\begin{equation}
\label{second.order.sol.2}
f(t)=c_1^\prime \Psi_\alpha^\nu(-\lambda_1 t)
+c_2^\prime \Psi_\alpha^{\nu+1}(-\lambda_1 t) ,
\end{equation}
where $c_1^\prime$ and $c_2^\prime$ are real constants. 
\end{enumerate}
\end{theorem}

\begin{proof} Let us write $a$ and $b$ in eq.\eqref{2nd.order.eq} as 
$a = \lambda_1 + \lambda_2$ and $b = \lambda_1 \lambda_2$ 
and factorize eq.\eqref{2nd.order.eq} as the pair of equations 
\begin{align}
\label{2nd.order.eq.1}
& \mathscr{D}_t^{(\alpha,\nu)} f(t) + \lambda_2 f(t) = g(t) , \\[1ex]
\label{2nd.order.eq.2}
& \mathscr{D}_t^{(\alpha,\nu)} g(t) + \lambda_1 g(t) = 0 . 
\end{align}
In Section~\ref{section.6} we saw that the solution of eq.\eqref{2nd.order.eq.2} is 
\begin{equation}
\label{solution.g}
    g(t) = c_0 \Psi_\alpha^\nu(-\lambda_1 t) , 
\end{equation}
where $c_0$ is an arbitrary constant. Applying the Mellin transform to eq.\eqref{2nd.order.eq.1} with the above expression for $g(t)$ and 
recalling eq.\eqref{mellin.derivative} we obtain 
\begin{equation}
    -\frac{\Gamma(s) \Gamma(1+\nu-s)}{\Gamma(s-1)\Gamma(2+\nu-\alpha-s)} F(s-1) 
    + \lambda_2 F(s) = c_0 \lambda_1^{-s} \frac{G(1+\nu)}{G(2-\alpha+\nu)} \, 
    \frac{\Gamma(s)\Gamma(2+\nu-\alpha-s)}{G(1+\nu-s)} . 
\end{equation}
which can be conveniently rewritten as 
\begin{equation}
\label{2nd.order.mellin}
-\Upsilon(s-1) + \lambda_2 \Upsilon(s) = c_0 \frac{G(1+\nu)}{G(2-\alpha-\nu)} \lambda_1^{-s} 
\end{equation}
with the definition 
\begin{equation}
    \Upsilon(s) = \frac{G(1+\nu-s)}{G(2+\nu-\alpha-s)}\, \frac{F(s)}{\Gamma(s)} .
\end{equation}

The solution of the homogeneous equation 
\begin{equation}
    -\Upsilon_{\scriptscriptstyle H}(s-1) + \lambda_2 \Upsilon_{\scriptscriptstyle H}(s) = 0 
\end{equation}
is easily seen to be 
\begin{equation}
    \Upsilon_{\scriptscriptstyle H}(s) = A_2 \lambda_2^{-s} ,
\end{equation}
where $A_2$ is an arbitrary constant. For the non-homogeneous equation we will look for a particular solution of the form 
\begin{equation}
    \Upsilon_{\scriptscriptstyle P}(s) = A_1 \lambda_1^{-s} ,
\end{equation}
which in eq.\eqref{2nd.order.mellin} gives 
\begin{equation}
    A_1(\lambda_2 - \lambda_1) = c_0 \frac{G(1+\nu)}{G(2-\alpha+\nu)} .
\end{equation}
Thus, if $\lambda_1 \neq \lambda_2$, we have a solution 
\begin{equation}
    \Upsilon_{\scriptscriptstyle P}(s) = \frac{c_0}{\lambda_2- \lambda_1} \frac{G(1+\nu)}{G(2-\alpha+\nu)}\lambda_1^{-s} . 
\end{equation}
So we can write the solution of eq.\eqref{2nd.order.mellin} as 
\begin{equation}
\Upsilon(s) = c_1 \frac{G(1+\nu)}{G(2-\alpha+\nu)} \lambda_1^{-s} 
+ c_2 \frac{G(1+\nu)}{G(2-\alpha+\nu)}\lambda_2^{-2} ,
\end{equation}
 where we defined $c_1$ and $c_2$ as 
\begin{equation}
    c_1 = \frac{c_0}{\lambda_2 - \lambda_1} , \qquad 
    c_2 = A_2 \frac{G(2-\alpha+\nu)}{G(1+\nu)} .
\end{equation}
In terms of $F(s)$ we have 
\begin{equation}
\begin{split} 
F(s) = & c_1 \frac{G(1+\nu)}{G(2-\alpha+\nu)} \lambda_1^{-s} \frac{\Gamma(s)G(2+\nu-\alpha-s)}{G(1+\nu-s)}\\[1ex]
& + c_2 \frac{G(1+\nu)}{G(2-\alpha+\nu)}\lambda_2^{-s} \frac{\Gamma(s)G(2+\nu-\alpha-s)}{G(1+\nu-s)},
\end{split}
\end{equation}
and taking inverse Mellin transform we obtain eq.\eqref{second.order.sol.1} 
when $\lambda_1 \neq \lambda_2$. 

If $\lambda_1 = \lambda_2$ we will look for a particular solution of the form 
\begin{equation}
    \Upsilon_{\scriptscriptstyle P}(s) = A_3 s \lambda_1^{-s} 
\end{equation}
which in eq.\eqref{2nd.order.mellin} gives 
\begin{equation}
    A_3 = \frac{c_0}{\lambda_1} \frac{G(1+\nu)}{G(2-\alpha+\nu)}
\end{equation}
and then 
\begin{equation}
\Upsilon(s) = c_1^\prime \frac{G(1+\nu)}{G(2-\alpha+\nu)} \lambda_1^{-s} 
+ c_0 \frac{G(1+\nu)}{G(2-\alpha+\nu)}s \lambda_1^{-(s+1)} ,
\end{equation}
and returning to $F(s)$, 
\begin{equation}
\begin{split} 
F(s) = & c_1^\prime \frac{G(1+\nu)}{G(2-\alpha+\nu)} \lambda_1^{-s} \frac{\Gamma(s)G(2+\nu-\alpha-s)}{G(1+\nu-s)}\\
& + c_0  \frac{G(1+\nu)}{G(2-\alpha+\nu)}\lambda_1^{-(s+1)} \frac{\Gamma(s+1)G(2+\nu-\alpha-s)}{G(1+\nu-s)} . 
\end{split}
\end{equation}
Finally, defining $c_2^\prime$ as 
\begin{equation}
c_0 = c_2^\prime \frac{\Gamma(1+\nu)}{\Gamma(2-\alpha+\nu)}
\end{equation}
and taking the inverse Mellin transform we obtain we obtain 
eq.\eqref{second.order.sol.2}. 
\end{proof}

\section{Stochastic applications}

\subsection{Scaling limits of CTRWs governed by the beta-weighted derivative}

The purpose of this section is to apply the results given in Section \ref{section.6} in order to define and study the limit of a continuous-time random walk
 with i.i.d. waiting times (between successive arrivals) $ J^{(\alpha,\nu)}$ with survival probability $\phi(t):=\mathbb{P}(J^{(\alpha,\nu)}>t)$ satisfying the relaxation equation 
\begin{equation} 
\mathscr{D}_t^{(\alpha,\nu)} \phi(t) = - \kappa \phi(t), \qquad t\geq 0, \label{ren1}
\end{equation}
under the initial condition $\phi(0)=1.$

As a consequence of Theorem \ref{main}, the inter-arrival-time distribution is equal to $\mathbb{P}(J^{(\alpha,\nu)}>t)=\Psi_\alpha^\nu(-\kappa t),$ for any $t \geq 0$. Since it is proved in Theorem \ref{second} that, for $\alpha \in(0,1)$, $\nu>-1$ and $\kappa >0$, the function $\phi(x) = \Psi_\alpha^\nu(-\kappa x)$ given in \eqref{eq.solution.final} is completely monotone (i.e. $(-1)^n\phi^{(n)}(x) \geq 0,$ for any $n \in \mathbb{N}$), then it can represent a suitable survival function. Moreover, it is evident from \eqref{eq.solution}, that $\Psi_{\alpha}^\nu(0)=1$.

Thus, we can define the renewal process $\left\{ N_{\alpha,\nu}(t)\right\}_{t\geq 0}$ as 
\begin{equation}N_{\alpha,\nu}(t):=\max\left\{ n \in \mathbb{N}_0: \sum_{i=1}^n J^{(\alpha,\nu)}_i \leq t\right\}, \qquad \alpha \in (0,1), \nu >-1,t \geq 0.\label{reninu}
\end{equation}

We evaluate the mean inter-arrival time of $N_{\alpha, \nu}$, as follows.
\begin{lem}
Let $\alpha \in (0,1)$ and $\nu >-1$, then the mean inter-arrival time of the process defined in \eqref{reninu} is given by 
\begin{equation}
    \mathbb{E}J^{(\alpha, \nu)}=\begin{cases}\displaystyle \frac{\Gamma(\nu)}{\kappa \Gamma(1-\alpha+\nu)}  , \; & \nu >0 , \\[1ex]
+\infty , & -1<\nu \leq 0 .\end{cases} \label{mit}
\end{equation}
\end{lem}
\begin{proof}
For any $\nu >0$, we consider
the definition \eqref{eq.solution} of $\Psi_\alpha^\nu(-\kappa x)$, so that 
\begin{equation}
\Psi_\alpha^\nu(-\kappa x) = -\frac{1}{\kappa}\frac{\Gamma(\nu)}{\Gamma(1-\alpha+\nu)}\frac{d\;}{dx}\Psi_{\alpha}^{\nu-1}(-\kappa x).  
\end{equation}
Therefore, we can write that
\begin{equation}
 \mathbb{E}J^{(\alpha, \nu)}=\int_0^\infty  \mathbb{P}(J^{(\alpha, \nu)}>x )dx=\int_0^\infty \Psi_\alpha^\nu(-\kappa x )\,dx = \frac{1}{\kappa} \frac{\Gamma(\nu)}{\Gamma(1-\alpha+\nu)}[1-\lim_{x\to \infty} \Psi_{\alpha}^{\nu-1}(-\kappa x)] ,
\end{equation}
which gives the result \eqref{mit}, by recalling the asymptotic expression \eqref{asypsi}.
\end{proof}

\begin{remark} 
The previous result can be alternatively obtained, in the case $\nu>0$, by recalling the Mellin transform of the function $\Psi_\alpha^\nu(-\kappa x)$, given in \eqref{mel}, with $f_0$  as in eq.\eqref{def.f_0}, as follows:
\begin{eqnarray}
 \mathbb{E}J^{(\alpha, \nu)}&=&\lim_{s \to 1^-}\int_0^\infty t^{s-1}\mathbb{P}(J^{(\alpha,\nu)}>t)dt=\lim_{s \to 1^-}\int_0^\infty t^{s-1}\Psi_{\alpha}^\nu(-\kappa t)dt \notag \\
 &=&\lim_{s \to 1^-} \kappa^{-s}\Gamma(s)\frac{G(1+\nu)}{G(2-\alpha+\nu)} \, \frac{G(2+\nu-\alpha-s)}{G(1+\nu-s)} \notag \\
 &=& \frac{1}{\kappa} \frac{G(1+\nu)}{G(2-\alpha+\nu)} \frac{G(1+\nu-\alpha)}{\nu}
\end{eqnarray} 
which coincides with \eqref{mit} after using 
$G(z+1) = \Gamma(z)G(z)$. The condition $\nu > 0$ follows from the 
condition of existence of the Mellin transform of $\Psi_\alpha^\nu(-\kappa t)$, 
that is, $0 < \operatorname{Re} s < 1 + \nu$.
\end{remark} 

Now we give the following definition.

\begin{definition}\label{def:ctrw}
    Let $X_{i},i=1,2,...$ be real, independent random variables, under the assumption that $X_{i}$ is independent of $N_{\alpha,\nu}$, for any $i=1,2,...$; then we define the CTRW $\left\{Y_{\alpha, \nu}(t)  \right\}_{t \geq 0},$ where 
\begin{equation}
Y_{\alpha, \nu}(t):=\sum_{i=1}^{N_{\alpha, \nu}(t)}X_{i}, \qquad t \geq 0.  \label{ya}
\end{equation}%
\end{definition}

In view of what follows, we recall that the $\beta$-stable subordinator $\left\{ \sigma^{(\beta)}(t)\right\}_{t \geq 0}$ is defined, for  $\beta \in (0,1)$ as the a.s. non-decreasing L\'{e}vy process with Laplace transform given in \eqref{subord}, and thus with Laplace exponent $\varphi(\lambda)=\lambda^\beta$.
We will denote by $\mathcal{L}_\beta:=\left\{ \mathcal{L}_\beta (t)\right\}_{t \geq 0}$ the inverse (or first passage time) of the $\beta$-stable subordinator, i.e. the process defined by
\[
\mathcal{L}_\beta(t):=\inf\left\{ s>0:\sigma^{(\beta)}(s)>t\right\}, \qquad t \geq 0.
\]

\begin{theorem}
\label{thm2} Let 
$Y_{\alpha, \nu}$ be the CTRW defined in Def. \ref{def:ctrw}, 
then, under the assumptions that $\nu \in (-1,0)$ and that the r.v.'s $X_i$ are i.i.d. random variables, for $i=1,2,...$, with $\mathbb{E} X=0$ and $\sigma^2_X:=\mathbb{E} X^2<\infty$,
we have the following convergence in the Skorokhod space $\mathbb{D}[0,\infty)$, with the $M_1$-topology, 
\begin{equation}\label{cmt}
\left\{ c^{-(1+\nu)/2}\sigma_X^{-1}Y_{\alpha,\nu}(ct)\right\} _{t\geq 0}\overset{%
M_{1}}{\Longrightarrow }\left\{ A^{1/2}_{\alpha, \nu}B(\mathcal{L}_{\nu+1}(t))\right\} _{t\geq 0},\qquad c\rightarrow +\infty ,
\end{equation}
where $$A_{\alpha, \nu}:=\frac{G(2+\nu-\alpha)\kappa^{\nu+1}\sin(-\pi \nu)}{G(1+\nu)G(1-\alpha)\pi},$$ and the inverse stable subordinator $\mathcal{L}_{\nu+1}$ is independent of the standard Brownian motion $B:=\left\{B(t)\right\}_{t \geq 0}$.
\end{theorem}

\begin{proof}
We apply the result given in \cite{MeerschaertSikorskii2012}, p.106, which proves that, if the i.i.d. waiting times $J_i$, $i=1,2,...$, are such that $\mathbb{P}(J>t) \sim B t^{-\beta}$, for $B>0$ and $\beta \in (0,1)$, then the following convergence in $\mathbb{D}[0,\infty)$, with the $M_1$-topology, holds for the corresponding rescaled renewal process $\left\{ N_\beta(t)\right\}_{t \geq 0}$:
\begin{equation} \label{m1}
\left\{c^{-\beta}N_\beta(ct)\right\}_{t \geq 0}\overset{M_1}{\Longrightarrow }\left\{a^{-\beta}\mathcal{L}_\beta(t)\right\}_{t \geq 0}, \qquad c \to \infty,
\end{equation}
where $a=(B\Gamma(1-\beta))^{1/\beta}$.
By applying the asymptotic expansion of the solution given in eq.\eqref{asypsi}, we have that
\begin{equation}\mathbb{P}(J^{(\alpha,\nu)}>t)=\Psi_\alpha^\nu(-\kappa t) \sim  
\frac{G(1+\nu)\Gamma(\nu+1)G(1-\alpha)(\kappa t)^{-\nu-1}}{G(2+\nu-\alpha)}=:C_{\alpha,\nu}t^{-\nu-1}, \label{m2}
\end{equation}
as $t \to \infty$. Thus, the convergence in \eqref{m1} holds for $\beta=\nu+1$ (which belongs to $(0,1)$ under the assumption that $\nu \in (-1,0)$) and for $a=(C_{\alpha,\nu}\Gamma(-\nu))^{1/(\nu+1)}$ Moreover, under the same assumption, the constant $C_{\alpha,\nu}$ appearing in \eqref{m2} is well-defined and non-negative. As a consequence, we can apply Theorem 4.2 in \cite{MeeSch} (see also \cite{MeerschaertSikorskii2012}, p.104), so that 
\begin{equation} \label{m3}\left\{c^{-1/2}\sum_{i=1}^{\lfloor ct \rfloor }X_i,c^{-(1+\nu)}N_{\alpha, \nu}(ct)\right\}_{t \geq 0}\overset{M_1}{\Longrightarrow } \left\{B(t),A_{\alpha,\nu}\mathcal{L}_{\nu+1}(t)\right\}_{t \geq 0}, \qquad c \to \infty,\end{equation}
in the product space $\mathbb{D}[0,\infty) \times \mathbb{D}[0, \infty)$. Finally, by the Continuous Mapping Theorem and by the independence of $N_{\alpha,\nu}$ and $X_j$, $j=1,2,...$, we get
\begin{equation}\notag
\left\{ c^{-(1+\nu)/2}\sigma_X^{-1}Y_{\alpha,\nu}(ct)\right\} _{t\geq 0}\overset{%
M_{1}}{\Longrightarrow }\left\{ B(A_{\alpha, \nu}\mathcal{L}_{\nu+1}(t))\right\} _{t\geq 0},\qquad c\rightarrow +\infty ,
\end{equation}
so that the convergence in \eqref{cmt} follows from the self-similarity property (of index $1/2$) of the Brownian motion.
\end{proof}

We note that, as a consequence of the previous result, the limiting process depends on the parameter $\alpha$ only through a constant, while the parameter $\nu$ influences also the diffusing velocity. Moreover, by the well-known results (see, for example, \cite{MAI}), the transition density $f_{\alpha,\nu}(\cdot,\cdot)$ of the limiting process 
defined as \[\mathbb{P}\left(\left.A_{\alpha, \nu}^{1/2}B(\mathcal{L}_{\nu+1}(t)) \in A \right\vert B(\mathcal{L}_{\nu+1}(0))=0\right)=\int_A f_{\alpha,\nu}(x,t)dx, \qquad t \geq 0, A \in \mathcal{B}(\mathbb R),\] satisfies the following problem for the time-fractional heat equation
\begin{equation}\label{caputo}
    \left\{ \begin{array}{l}
\sideset{_{\scriptscriptstyle{\textrm C}}^{}}{_t^{\nu+1}}{\operatorname{\mbox D}} f(x,t)=\frac{A_{\alpha,\nu}}{2}\frac{\partial^2}{\partial x^2} f(x,t)  \qquad x \in \mathbb{R}, t \geq 0, \\
f(x,0)=\delta(x), \\
\lim_{|x| \to \infty}f(x,t)=0, \quad \lim_{|x| \to \infty}\frac{\partial}{\partial x}f(x,t)=0.
    \end{array}\right.
\end{equation}
where $\sideset{_{\scriptscriptstyle{\textrm C}}^{}}{_t^{\beta}}{\operatorname{\mbox D}}$ is the Caputo fractional derivative of order $\beta=\nu+1$, under the usual initial and boundary conditions. Indeed, by recalling the Laplace transform of the inverse stable subordinator
\begin{equation}
    \mathbb Ee^{-\lambda \mathcal{L}_\beta(t)}=E_\beta(-\lambda t^\beta), \qquad \lambda>0,\label{ML}\end{equation}
where $E_\beta(x):=\sum_{n=0}^\infty x^n/\Gamma(\beta n+1)$ is the Mittag-Leffler function, the characteristic function of $\left\{A_{\alpha,\nu}B(\mathcal{L}_{\nu+1}(t))\right\}_{t \geq 0}$ is given by
\begin{equation}\mathbb Ee^{i\xi A_{\alpha,\nu}^{1/2}B(\mathcal{L}_{\nu+1}(t))}=E_{\nu+1}\left(-\xi^2A_{\alpha,\nu}t^{\nu+1}/2 \right), \qquad \xi \in \mathbb R, t \geq 0,\label{ch}\end{equation}
which coincides with the Fourier transform of the solution to \eqref{caputo}.

It is well-known that the process obtained in the limit represents a sub-diffusion model, since $var(A_{\alpha,\nu}^{1/2}B(\mathcal L_{\nu+1}(t))=A_{\alpha,\nu}\mathbb E\mathcal L_{\nu+1}(t)=A_{\alpha,\nu}t^{1+\nu}/\Gamma(\nu+2),$ where $\nu+1<1.$ 

Moreover, by \eqref{ML}, the following self-similarity holds:
\[\left\{\mathcal L_\beta(ct)\right\}\overset{f.d.d.}{=}\left\{c^\beta\mathcal L_\beta(t)\right\},\]
where $\overset{f.d.d.}{=}$ denotes equality of finite-dimensional distributions. As a consequence, we also have that
\[\left\{B(\mathcal L_{\nu+1}(ct))\right\}\overset{f.d.d.}{=}\left\{c^{(\nu+1)/2}B(\mathcal L_{\nu+1}(t))\right\}.\]

Thus the spatial scale growing rate is $t^{(\nu+1)/2}$, slower than $t^{1/2}$ of the standard diffusion, for $\nu \in (-1,0)$.

\subsection{Time-changed Brownian motions associated with the beta-weighted derivative}
In order to obtain a model of diffusion governed by a fractional equation with time-derivative $\mathscr{D}_t^{(\alpha,\nu)}$, we must follow a different approach. By resorting again to the complete monotonicity of $\phi(x) = \Psi_\alpha^\nu(-\kappa x)$, we define the following stochastic process.

\begin{definition}\label{def}
    Let $X_{\alpha, \nu}$ be a random variable with Laplace transform of the density $f_{X_{\alpha, \nu}}(\cdot)$ given in \eqref{LTpsi}, then we define the process $Y_{\alpha,\nu}:=\left\{ Y_{\alpha,\nu}(t) \right\}_{t \geq 0}$ as follows

\begin{equation}
    ( Y_{\alpha,\nu}(t_1),...,  Y_{\alpha,\nu}(t_n)):=( t_1 X_{\alpha,\nu},...,  t_nX_{\alpha,\nu}),
\end{equation}
for any $n \in \mathbb{N}$ and $t_1,...,t_n \geq 0$. Let now $B:=\left\{ B(t) \right\}_{t \geq 0}$ be a standard Brownian motion independent of $Y_{\alpha,\nu}$, then we consider the time-changed Brownian motion $B_{\alpha,\nu}:=\left\{ B_{\alpha,\nu}(t) \right\}_{t \geq 0}$, as
\begin{equation} \label{fdd}( B_{\alpha,\nu}(t_1),...,  B_{\alpha,\nu}(t_n)):=(B(Y_{\alpha,\nu}( t_1) ),...,  B(Y_{\alpha,\nu}(t_n))=(B( t_1 X_{\alpha,\nu}),...,  B(t_nX_{\alpha,\nu})),\end{equation}
for any $0 \leq t_1 <...<t_n,$ $n \geq 1.$

\end{definition}
In view of \eqref{LTpsi}, the joint characteristic function reads
\begin{eqnarray}\label{chf}
 \mathbb{E}e^{i \sum_{k=1}^n\xi_k B_{\alpha,\nu}(t_k)}&=& \mathbb{E}\left[\mathbb{E}\left(e^{i \sum_{k=1}^n\xi_k B(t_k X_{\alpha,\nu})}\vert X_{\alpha,\nu}\right) \right] \\
 &=&\mathbb{E}e^{-\frac{X_{\alpha,\nu}}{2}\sum_{k,j=1}^n\xi_k\xi_j(t_k \wedge t_j)} \notag \\
 &=&\Psi_\alpha^\nu\left(-\frac{1}{2}\sum_{k,j=1}^n\xi_k\xi_j(t_k \wedge t_j) \right), 
\end{eqnarray}
for $\xi_,...,\xi_n \in \mathbb{R}$ and $0 \leq t_1 <...<t_n,$ $n \geq 1.$

It can be checked that the following equality of f.d.d.'s holds, by the scaling property of the Brownian motion:
\begin{equation}\left\{B_{\alpha, \nu}(t) \right\}_{t\geq 0}\overset{f.d.d.}{=} \left\{\sqrt{X_{\alpha,\nu}}B(t)\right\}_{t \geq 0,}\label{fdd}\end{equation}
under the assumption that $X_{\alpha,\nu}$ and $\{B(t)\}_{t \geq 0}$ are independent. 

By conditioning, it is easy to see from \eqref{fdd} that $\mathbb{E}B_{\alpha,\nu}(t)=0$, for any $t \geq 0$; by recalling \eqref{mom} we can evaluate its auto-covariance function, as follows
\[\mathbb{E}\left(B_{\alpha,\nu}(t_k)B_{\alpha,\nu}(t_j)\right)=(t_k \wedge t_j)\mathbb{E}X_{\alpha, \nu}= \frac{G(3+\nu-\alpha)(t_k \wedge t_j)}{G(2+\nu)G(2+\nu-\alpha)}=\frac{\Gamma(2+\nu-\alpha)(t_k \wedge t_j)}{\Gamma(1+\nu)},\] 
for $t_j,t_k \geq 0$, where, in the last step, we have applied the property of the $G$-function (see Appendix \ref{appendix.Barnes}, formula \eqref{barnes.g}). This clearly shows that $var(B_{\alpha,\nu}(t)) \simeq t$, and thus the process is actually a standard diffusion, even though, as we prove below, its one-dimensional density satisfies a time-fractional diffusion equation with operator $\mathscr{D}_t^{(\alpha,\nu)}$. 
It is immediate to check that the one-dimensional density of the process $Y_{\alpha,\nu}$ is given by $f_{Y_{\alpha,\nu}}(y,t)=\frac{1}{t}f_{X_{\alpha,\nu}}(y/t)$, for $y \in \mathbb{R}^+$, $t \geq 0.$ As a consequence, the density of $B_{\alpha,\nu}$ can be obtained as
\begin{equation}\label{td}
    f_{B_{\alpha,\nu}}(z,t)=\int_0^\infty \frac{e^{-z^2/2y}}{\sqrt{2 \pi y}t}f_{X_{\alpha,\nu}}(y/t)dy, \qquad z \in \mathbb{R}, t > 0,
\end{equation}
and it satisfies the following problem:
\begin{equation}\label{pde}
    \left\{ \begin{array}{l}
\mathscr{D}_t^{(\alpha,\nu)} f(x,t)=\frac{1}{2}\frac{\partial^2}{\partial x^2} f(x,t)  \qquad x \in \mathbb{R}, t \geq 0, \\
f(x,0)=\delta(x), \\
\lim_{|x| \to \infty}f(x,t)=0, \quad \lim_{|x| \to \infty}\frac{\partial}{\partial x}f(x,t)=0.
    \end{array}\right.
\end{equation}
Indeed, by taking the Fourier transform of \eqref{pde}, w.r.t. the space argument, we get
\begin{equation}
    \left\{ \begin{array}{l}
\mathscr{D}_t^{(\alpha,\nu)} \tilde{f}(\xi,t)=-\frac{\xi^2}{2}\Tilde{f}(\xi,t),\\
\tilde{f}(\xi,0)=1,
\end{array}  \right.
\end{equation}
for $\xi \in \mathbb{R}$. The latter is satisfied by the characteristic function of $B_{\alpha,\nu}$, i.e. $ \mathbb{E}e^{i \xi B_{\alpha,\nu}(t)}=\Psi_\alpha^\nu\left(-\frac{1}{2}\xi^2t \right)$ (see \eqref{chf} for $n=1$).  Finally, it is easy to see that the latter does not coincide with the characteristic function of  
$\left\{A_{\alpha,\nu}B(\mathcal{L}_{\nu+1}(t))\right\}_{t \geq 0}$, given in \eqref{ch}.

\begin{remark}
The process $\left\{B_{\alpha,\nu}(t) \right\}_{t\geq 0}$ can be considered a generalization of the so-called grey Brownian motion $\left\{B_{\beta}(t) \right\}_{t\geq 0}$ (introduced in \cite{SCH} and studied by many other authors), whose characteristic function reads $ \mathbb{E}e^{i \sum_{k=1}^n\xi_k B_{\beta}(t_k)}=E_\beta\left(-\frac{1}{2}\sum_{k,j=1}^n\xi_k\xi_j(t_k \wedge t_j) \right)$, for $\xi_,...,\xi_n \in \mathbb{R}$ and $0 \leq t_1 <...<t_n,$ $n \geq 1.$

Indeed for the grey Brownian motion, the following equality of f.d.d.'s holds (analogous to \eqref{fdd}):
\begin{equation}\left\{B_{\beta}(t) \right\}_{t\geq 0}\overset{f.d.d.}{=} \left\{\sqrt{Z_{\beta}}B(t)\right\}_{t \geq 0,}\end{equation}
where $Z_{\beta}$ is a r.v. with Wright distribution (see \cite{MUR}, for details) and under the independence assumption.

\end{remark}

\section*{Acknowledgments} 
Luisa Beghin acknowledges financial support under NRRP, Mission 4, Component 2, 
Investment 1.1, Call for tender No. 104 published on 2.2.2022 by the Italian MUR, 
funded by the European Union - NextGenerationUE - Project Title ``Non-Markovian 
Dynamics and Non-local Equations", 202277N5H9 - CUP: D53D23005670006. 

Nikolai Leonenko (NL) was partially supported under the ARC Discovery Grant DP220101680 (Australia), Croatian Scientific Foundation (HRZZ) grant IP-2022-10-8081, Retreat programme (2026) at Isaac Newton Institute for Mathematical Sciences, Cambridge and Taith Research Mobility grant (Wales, Cardiff University). Also, NL would like to thank University of Rome ``La Sapienza" for hospitality and financial support as Visiting Professor (June 2024) where the paper was initiated.

Jayme Vaz would like to thank the support of FAPESP (process 24/17510-6), and Sapienza Universit\`a di Roma for the hospitality as Visiting Professor during the development of this work.

\section*{Competing interests}The authors have no conflicts of interest to declare that are relevant to the content of this article.

\begin{appendices}

\section{Barnes $\boldsymbol{G}$-function}
\label{appendix.Barnes}

In view of what follows we recall the definition and main properties of the Barnes $G$-function: 
\begin{equation}
\label{def.barnes.g}
G(z+1) = (2\pi)^{z/2} \exp\left(-\frac{z+z^2(1+\gamma)}{2}\right) 
\prod_{k=1}^\infty \left[\left(1+\frac{z}{k}\right)^k \exp\left(\frac{z^2}{2k}-z\right)\right] , \qquad z \in \mathbb C,
\end{equation}
where $\gamma$ is the Euler-Mascheroni constant (\cite{BAR}). 
The function $G(z)$ is holomorphic for $z \in \mathbb{C}$, has zeros 
at $z = -n$ ($n=0,1,2,\ldots$) of order $n+1$, 
 and it satisfies   
\begin{align}
\label{barnes.g}
& G(z+1) = \Gamma(z) G(z) \\[1ex]
& G(1) = 1 
\end{align}
The plot of $G(z)$ is in Figure~\ref{fig.1}. 

\begin{figure}[hbt]
\begin{center}
\includegraphics[width=12cm]{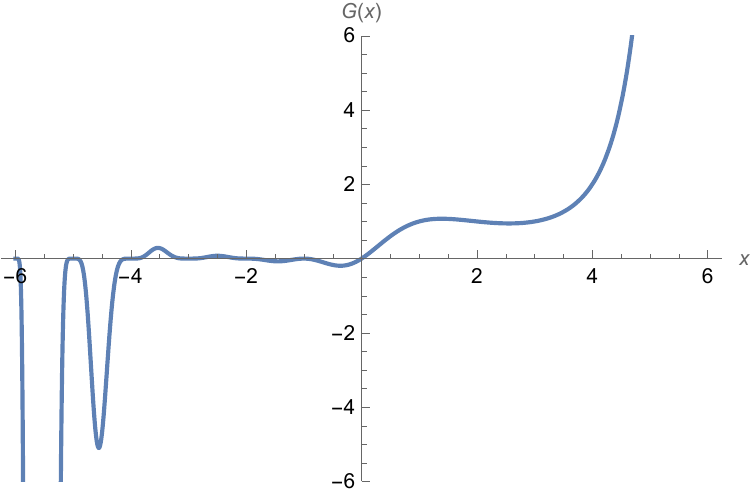}
\caption{Barnes $G$-function.\label{fig.1}}
\end{center}
\end{figure}

From eq.\eqref{def.barnes.g} it follows that 
\begin{equation}
\log{G(z+1)} = \frac{1}{2}(\log{2\pi} - 1)z - \frac{(1+\gamma)}{2} z^2 
+ \sum_{n=3}^\infty (-1)^{n-1}\zeta(n-1) \frac{z^n}{n} ,
\end{equation}
where $\zeta(n) = \sum_{k=1}^\infty k^{-n}$ is the zeta function. 
Writing $G(z) = G(z+1)/\Gamma(z)$, using the power series for $1/\Gamma(z)$ and for 
$G(z+1) = 1 + \log{G(z+1)} + (\log{G(z+1)})^2/2 + \cdots$ we obtain 
\begin{equation}
\label{barnes.g.series}
    G(z) = z + \left(\gamma + \frac{1}{2}(\log{2\pi} -1)\right) z^2 + 
    \left[\frac{1}{8}(\log{2\pi}-1)^2 - \frac{(1+\gamma)}{2} + 
    \frac{\gamma}{2}(\log{2\pi}-1)\right] z^3 + \mathcal{O}(z^4) 
\end{equation}
for $|z| < 1$.

\section{Auxiliary results}

We want to show that 
\begin{equation}
\label{absolute.value.barnes}
|G(z)| = |G(x)| \exp\left(y^2\frac{1+\gamma}{2}\right) 
\sqrt{1+\frac{y^2}{x^2}}\sqrt{\prod_{k=1}^\infty 
\left(1+\frac{y^2}{(x+k)^2}\right)^{k+1} \exp\left(-\frac{y^2}{k}\right)} 
\end{equation}
where $z = x+iy$. 
From eq.\eqref{def.barnes.g} and eq.\eqref{barnes.g}, we write the definition 
of $G(z)$ as 
$$
G(z) = \frac{(2\pi)^{z/2}}{\Gamma(z)} 
\exp\left(-\frac{z+z^2(1+\gamma)}{2}\right) \prod_{k=1}^\infty 
\left[\left(1+\frac{z}{k}\right)^k \exp\left(\frac{z^2}{2k}-z\right)\right] .
$$
Evaluating $|G(z)|^2 = G(z)\overline{G(z)} = G(z)G(\bar{z})$, we obtain 
\begin{equation}
\begin{aligned}
\label{eq.barnesg.aux}
|G(z)|^2 = & \frac{1}{|\Gamma(z)|^2} \left[ 
(2\pi)^{x/2}\exp\left(-\frac{x+x^2(1+\gamma)}{2}\right) 
\prod_{k=1}^\infty \left(1+\frac{x}{k}\right)^k 
\exp\left(\frac{x^2}{2k}-x\right)\right]^2 \\[1ex]
& \cdot \operatorname{e}^{y^2(1+\gamma)} \prod_{k=1}^\infty \left(1+\frac{y^2}{(x+k)^2}\right)^k 
\exp\left(-\frac{y^2}{k}\right) .
\end{aligned}
\end{equation}
Recalling Weierstrass definition of the gamma function, that is, 
$$
\frac{1}{\Gamma(z)} = z \operatorname{e}^{\gamma z} \prod_{k=1}^\infty \left(1+\frac{z}{k}\right) \operatorname{e}^{-z/n} , 
$$
which gives 
\begin{equation}
\label{abs.gamma.complex}
\frac{1}{|\Gamma(z)|^2} = \frac{1}{|\Gamma(x)|^2} 
\prod_{k=0}^\infty \left(1 + \frac{y^2}{(x+k)^2}\right) ,
\end{equation}
then we get from eq.\eqref{eq.barnesg.aux} that 
$$
|G(z)|^2 = |G(x)|^2 \left(1 + \frac{y^2}{x^2}\right)
\prod_{k=1}^\infty \left(1 + \frac{y^2}{(x+k)^2}\right) 
\operatorname{e}^{y^2(1+\gamma)} \prod_{k=1}^\infty \left(1+\frac{y^2}{(x+k)^2}\right)^k 
\exp\left(-\frac{y^2}{k}\right) ,
$$
which gives eq.\eqref{absolute.value.barnes}. 

\section{The integrals along $C_\uparrow/C_\downarrow$}

Let us write $\operatorname{Re}s = \xi = $ and $\operatorname{Im}s = \eta$. 
Then we can write eq.\eqref{exp.log.integrand} as 
\begin{equation}
\label{exp.log.integrand.2}
\begin{aligned}
& \log\left|(\kappa x)^{-s} \frac{\Gamma(s) G(2-s-\alpha)}{  G(1-s)}\right|  =   
\alpha\xi\log{\sqrt{\xi^2+\eta^2}}  -\xi \log|\kappa x| + \eta \operatorname{arg}(\kappa x) \\[1ex] 
& \qquad -\alpha(\theta\eta + \xi) -(1-\alpha)\pi |\eta|   
+ \frac{(\alpha+\nu)(\alpha+\nu-2)}{2}\log{\sqrt{\xi^2+\eta^2}} + \mathcal{O}(1) ,
\end{aligned}
\end{equation}
with $\xi = R\cos\theta$ and $\eta = R\sin\theta$. For $s \in C_\uparrow$ [resp. 
$s \in C_\downarrow$] we have $s = \xi + i \tau$ [resp. $s = \xi - i\tau$] 
with $0 < \xi < \nu + m + 1/2$. For $\tau \to \infty$ the leading term in 
eq.\eqref{exp.log.integrand.2} for $s \in C_\uparrow$ is 
\begin{equation}
\label{tau.1}
\tau(\operatorname{arg}(\kappa x) - \alpha\theta - (1-\alpha)\pi) 
\end{equation}
with $\theta \to \pi/2$ for $\tau \to \infty$, and for $s \in C_\downarrow$ the 
leading term is 
\begin{equation}
\label{tau.2}
-\tau(\operatorname{arg}(\kappa x) - \alpha\theta + (1-\alpha)\pi) 
\end{equation}
with $\theta \to -\pi/2$ for $\tau \to \infty$. The integral along $C_\uparrow$ 
will vanishes for $\tau \to \infty$ if 
$$
(\operatorname{arg}(\kappa x) - \alpha\frac{\pi}{2} - (1-\alpha)\pi) < 0 \; \quad \Rightarrow \quad
\; \operatorname{arg}(\kappa x) < \left(1-\frac{\alpha}{2}\right) \pi ,
$$
and the integral along $C_\downarrow$ will vanishes for $\tau \to \infty$ if 
$$
-(\operatorname{arg}(\kappa x) + \alpha\frac{\pi}{2} + (1-\alpha)\pi) < 0 
\; \quad \Rightarrow \quad \; \operatorname{arg}(\kappa x) > - \left(1-\frac{\alpha}{2}\right) \pi . 
$$
Thus the integrals along $C_\uparrow$ and $C_\downarrow$ will vanish for 
$\tau \to \infty$ if 
$$
|\operatorname{arg}\kappa x | < \left(1-\frac{\alpha}{2}\right) \pi . 
$$

\section{The integral along $C^\prime$}

Let us denote the integral along $C^\prime$ by $J_{\nu,m}(\alpha)$, that is, 
\begin{equation}
\begin{aligned}
& J_{\nu,m}(\alpha) =   -\frac{1}{2\pi i}\frac{G(1+\nu)}{G(2+\nu-\alpha)}
\int_{(\nu+m+1/2)-i\infty}^{(\nu+m+1/2)+i\infty}(\kappa x)^{-s}  \Gamma(s) \frac{G(2+\nu-\alpha-s)}{G(1+\nu-s)} \, ds \\[1ex]
& \quad = - \frac{1}{2\pi i}\frac{G(1+\nu)}{G(2+\nu-\alpha)}
\int_{(\nu+m+1/2)-i\infty}^{(\nu+m+1/2)+i\infty}(\kappa x)^{-s}  \Gamma(s)\Gamma(1+\nu-s) \frac{G(2+\nu-\alpha-s)}{G(2+\nu-s)} \, ds 
\end{aligned}
\end{equation}
Taking $s = (\nu + m + 1/2) + i \eta$ and using the triangle inequality for integrals, we have 
\begin{equation}
\begin{aligned}
| J_{\nu,m}(\alpha)|\leq & \frac{1}{2\pi}\left|\frac{G(1+\nu)}{G(2+\nu-\alpha)}\right| 
|\kappa x|^{-(\nu+m+1/2)} \\[1ex]
& \cdot \int_{-\infty}^\infty 
\left|\Gamma\left({\textstyle \nu+m+\frac{1}{2}+i\eta}\right)\right| 
\left|\Gamma\left({\textstyle  \frac{1}{2}-m-i\eta}\right)\right| 
\left|\frac{G\left(\frac{3}{2}-m-\alpha-i\eta\right)}{G\left(\frac{3}{2}-m-i\eta\right)}\right| 
\, ds .
\end{aligned}
\end{equation}
Using eq.\eqref{abs.gamma.complex} we can write 
\begin{equation}
\label{eq.gamma.aux.1}
\left|\Gamma\left({\textstyle \nu+m+\frac{1}{2}+i\eta}\right)\right| 
\left|\Gamma\left({\textstyle  \frac{1}{2}-m-i\eta}\right)\right|  = 
\frac{\left|\Gamma\left({\textstyle \nu+m+\frac{1}{2}}\right)\right| 
\left|\Gamma\left({\textstyle  \frac{1}{2}-m}\right)\right| }{ 
\sqrt{{\displaystyle \prod_{k=0}^\infty} \left(1 + \frac{\eta^2}{(k+\nu+\frac{1}{2}+m)^2}\right)  
 \left(1 + \frac{\eta^2}{(k+\frac{1}{2}-m)^2}\right)} } ,
\end{equation}
and using eq.\eqref{absolute.value.barnes}, 
\begin{equation}
\left|\frac{G\left(\frac{3}{2}-m-\alpha-i\eta\right)}{G\left(\frac{3}{2}-m-i\eta\right)}\right| 
= \left|\frac{G\left(\frac{3}{2}-m-\alpha\right)}{G\left(\frac{3}{2}-m\right)}\right| 
\frac{\sqrt{1+\frac{\eta^2}{\left(\frac{3}{2}-m-\alpha\right)^2}}}{\sqrt{1+\frac{\eta^2}{\left(\frac{3}{2}-m\right)^2}}}
\frac{\sqrt{{\displaystyle \prod_{k=1}^\infty} \left( 1+\frac{\eta^2}{\left(\frac{3}{2}-m-\alpha+k\right)^2}\right)^{k+1}}}{\sqrt{{\displaystyle \prod_{k=1}^\infty} \left( 1+\frac{\eta^2}{\left(\frac{3}{2}-m+k\right)^2}\right)^{k+1}}}
\end{equation}
Since $\alpha \leq 1$, we have 
$$
\frac{3}{2}-m -\alpha \geq \frac{1}{2}-m , 
$$
and  we can write 
\begin{equation}
\sqrt{1+\frac{\eta^2}{\left(\frac{3}{2}-m-\alpha\right)^2}} \leq 
\sqrt{1+\frac{\eta^2}{\left(\frac{1}{2}-m\right)^2}}
\end{equation}
and
\begin{equation}
\begin{aligned}
{\displaystyle \prod_{k=1}^\infty} \left( 1+\frac{\eta^2}{\left(\frac{3}{2}-m-\alpha+k\right)^2}\right)^{k+1} & \leq 
{\displaystyle \prod_{k=1}^\infty} \left( 1+\frac{\eta^2}{\left(\frac{1}{2}-m+k\right)^2}\right)^{k+1} \\[1ex]
& = \left(1+\frac{\eta^2}{\left(\frac{3}{2}-m\right)^2}\right)^2 
{\displaystyle \prod_{k=1}^\infty} \left( 1+\frac{\eta^2}{\left(\frac{3}{2}-m+k\right)^2}\right)^{k+2} .
\end{aligned}
\end{equation}
Thus 
\begin{equation}
\sqrt{\frac{{\displaystyle \prod_{k=1}^\infty} \left( 1+\frac{\eta^2}{\left(\frac{3}{2}-m-\alpha+k\right)^2}\right)^{k+1}}{{\displaystyle \prod_{k=1}^\infty} \left( 1+\frac{\eta^2}{\left(\frac{3}{2}-m+k\right)^2}\right)^{k+1}} } \leq 
\left(1+\frac{\eta^2}{\left(\frac{3}{2}-m\right)^2}\right) \sqrt{  
{\displaystyle \prod_{k=1}^\infty} \left( 1+\frac{\eta^2}{\left(\frac{3}{2}-m+k\right)^2}\right)} ,
\end{equation}
and 
\begin{equation}
\label{eq.barnes.g.aux.1}
\begin{aligned}
\left|\frac{G\left(\frac{3}{2}-m-\alpha-i\eta\right)}{G\left(\frac{3}{2}-m-i\eta\right)}\right| 
& \leq   \left|\frac{G\left(\frac{3}{2}-m-\alpha\right)}{G\left(\frac{3}{2}-m\right)}\right| 
\sqrt{1+\frac{\eta^2}{\left(\frac{1}{2}-m\right)^2}} 
\sqrt{  
{\displaystyle \prod_{k=0}^\infty} \left( 1+\frac{\eta^2}{\left(\frac{3}{2}-m+k\right)^2}\right)} \\[1ex]
&  = \left|\frac{G\left(\frac{3}{2}-m-\alpha\right)}{G\left(\frac{3}{2}-m\right)}\right|
\sqrt{  
{\displaystyle \prod_{k=0}^\infty} \left( 1+\frac{\eta^2}{\left(\frac{1}{2}-m+k\right)^2}\right)}
\end{aligned}
\end{equation}
Using eq.\eqref{eq.gamma.aux.1} and eq.\eqref{eq.barnes.g.aux.1}, we have 
\begin{equation}
|J_{\nu,m}(\alpha)| \leq  \Theta(\nu,m,\alpha)  |\kappa x|^{-(\nu+m+1/2)} 
\int_{-\infty}^\infty \frac{d\eta}{\sqrt{{ \prod_{k=0}^\infty} 
\left(1 + \frac{\eta^2}{\left( k + \nu + \frac{1}{2}+m\right)^2}\right)}} , 
\end{equation}
where 
$$
 \Theta(\nu,m,\alpha) = \frac{1}{2\pi}\left| 
\frac{G(1+\nu) \Gamma\left(\nu+ m + \frac{1}{2}\right) 
\Gamma\left(\frac{1}{2}-m\right) G\left(\frac{3}{2}-m-\alpha\right)}{G(2+\nu-\alpha) G\left(\frac{3}{2}-m\right)} \right| . 
$$
We can simplify the above integral by taking $N = \lceil \nu + m \rceil$. Then 
$$
{ \prod_{k=0}^\infty} 
\left(1 + \frac{\eta^2}{\left( k + \nu + \frac{1}{2}+m\right)^2}\right) \geq 
{ \prod_{k=0}^\infty} 
\left(1 + \frac{\eta^2}{\left( k + N +\frac{1}{2}\right)^2}\right) = 
\frac{{ \prod_{j=0}^\infty} 
\left(1 + \frac{\eta^2}{\left( j + \frac{1}{2}\right)^2}\right)}{{ \prod_{j=0}^{N-1}} 
\left(1 + \frac{\eta^2}{\left( j + \frac{1}{2}\right)^2}\right)}
$$
and we can write 
$$
\int_{-\infty}^\infty \frac{d\eta}{\sqrt{{ \prod_{k=0}^\infty} 
\left(1 + \frac{\eta^2}{\left( k + \nu + \frac{1}{2}+m\right)^2}\right)}} \leq 
\int_{-\infty}^\infty \sqrt{\frac{P_{2N-2}(\eta)}{\cosh{\pi \eta}}} \, d\eta = \mathcal{I}(N), 
$$
where $P_{2N-2}(\eta)$ is the polynomial of degree $2N-2$ given by 
$$
P_{2N-2}(\eta) = { \prod_{j=0}^{N-1}} 
\left(1 + \frac{\eta^2}{\left( j + \frac{1}{2}\right)^2}\right) 
$$
and we have used 
$$
\cosh{\pi \eta} = { \prod_{j=0}^\infty} 
\left(1 + \frac{\eta^2}{\left( j + \frac{1}{2} \right)^2}\right) . 
$$
The integral $\mathcal{I}(N)$ is clearly convergent and finally we have
$$
|J_{\nu,m}(\alpha)| \leq  \mathcal{I}(N) \Theta(\nu,m,\alpha)  |\kappa x|^{-(\nu+m+1/2)} . 
$$

\section{Evaluation of residues}
\label{residues}
\subsection{Evaluation of eq.\eqref{residue.1}}
Let us evaluate 
\begin{equation}
\rho_k = \underset{t=0}{\operatorname{Res}} \left[\frac{(\kappa x)^{-t} \Gamma(\nu+k+t)G(2-\alpha-k-t)}{  G(1-k-t)} \right] 
\end{equation} 
for $k = 1$ and $k=2$. For this we will use the following expressions 
\begin{align}
\label{barnesg.1}
& G(z) = z + \left(\gamma + \frac{1}{2}(\log{2\pi}-1)\right) z^2 + 
\mathcal{O}(z^3) , \\[1ex]
\label{barnesg.2}
& G(z-1) = -z^2 -\frac{1}{2}(4\gamma + \log{2\pi}-3)z^3 + \mathcal{O}(z^4), \\[1ex]
\label{barnesg.3}
& G^\prime(z) = G(z)\left[\frac{1}{2}\left(1+\log{2\pi}\right) - z + (z-1)\psi(z)\right] ,
\end{align}
where eq.\eqref{barnesg.1} is eq.\eqref{barnes.g.series} reproduced here for convenience, eq.\eqref{barnesg.2} is obtained from eq.\eqref{barnesg.1} using 
$G(z-1) = (z-1) G(z)/\Gamma(z)$ and the series for $1/\Gamma(z)$, that is 
\begin{equation}
    \frac{1}{\Gamma(z)} = z + \gamma z^2 +\left(\frac{\gamma^2}{2} - \frac{\pi^2}{12}\right) z^3 + \mathcal{O}(z^4) , 
\end{equation}
and eq.\eqref{barnesg.3} is obtained from (2.17) in \cite{Choi}. 

For $k=1$ we have 
\begin{equation}
\begin{split}
& \frac{(\kappa x)^{-t} \Gamma(\nu+1+t)G(1-\alpha-t)}{G(-t)} \\[1ex]
& = 
\frac{(1-t\log{\kappa x} + \ldots)(\Gamma(\nu+1)+\Gamma^\prime(\nu+1)t+\ldots)(G(1-\alpha)-G^\prime(1-\alpha)t+\ldots)}{-t\left[1-(\gamma+(1/2)(\log{2\pi}-1) t + \ldots\right]}\\[1ex]
& = -\frac{\Gamma(\nu+1)\Gamma(1-\alpha)}{t} + \mathcal{O}(1) 
\end{split}
\end{equation}
and so 
\begin{equation}
    \rho_1 = -\Gamma(\nu+1)G(1-\alpha) . 
\end{equation}

For $k = 2$ we have 
\begin{equation}
\begin{split}
& \frac{(\kappa x)^{-t} \Gamma(\nu+2+t)G(-\alpha-t)}{G(-1-t)} \\[1ex]
& = 
\frac{(1-t\log{\kappa x} + \cdots)(\Gamma(\nu+2)+\Gamma^\prime(\nu+2)t+\ldots)(G(-\alpha)-G^\prime(-\alpha)t+\cdots)}{-t^2\left[1-(\gamma+(1/2)(4\gamma-3+\log{2\pi}) t + \cdots\right]}\\[1ex]
& = -\frac{\Gamma(\nu+2)\Gamma(-\alpha)}{t^2} - \frac{1}{t} \big[ -\Gamma(\nu+2)G(-\alpha) \log{\kappa x} + \Gamma^\prime(\nu+2) G(-\alpha) \\[1ex]
& \quad - G^\prime(-\alpha)\Gamma(\nu+2) + 
\frac{1}{2}(4\gamma-3+\log{2\pi}) \Gamma(\nu+2)G(-\alpha)\big] + \mathcal{O}(1) 
\end{split}
\end{equation}
and using eq.\eqref{barnesg.3} for $G^\prime(-\alpha)$ and expressing $\Gamma^\prime(\cdot)$ in terms
of the digamma function, we obtain 
\begin{equation}
\begin{split}
& \frac{(\kappa x)^{-t} \Gamma(\nu+2+t)G(-\alpha-t)}{G(-1-t)} \\[1ex]
& = -\frac{\Gamma(\nu+2)\Gamma(-\alpha)}{t^2} \\[1ex]
& \quad + \frac{\Gamma(\nu+2)G(-\alpha)}{t} 
\big[\log{\kappa x} -\psi(\nu+2) +\alpha - (\alpha+1)\psi(-\alpha) - 
2\gamma+2\big]+ \mathcal{O}(1) 
\end{split}
\end{equation}
and so 
\begin{equation}
\rho_2 = \Gamma(\nu+2)G(-\alpha)
\big[\log{\kappa x} -\psi(\nu+2) +\alpha - (\alpha+1)\psi(-\alpha) - 
2\gamma+2\big] .
\end{equation} 

\subsection{Evaluation of eq.\eqref{residue.case.i}, eq.\eqref{residue.case.ii} and eq.\eqref{residues.case.iii}}

Let us write eq.\eqref{residue.case.i} as 
\begin{equation}
\rho_1^{\text{(i)}} = (\kappa z^{-1})^{-(1+\mu)}
\underset{t=0}{\operatorname{Res}} \, (\kappa z)^{-t}\Gamma(-t)
\Gamma(1+\mu+t)\frac{G(1+\nu-\mu-\alpha-t)}{G(\nu-\mu-t)} . 
\end{equation}
Recalling that 
\begin{equation}
\label{gamma.series}
    \Gamma(z) = \frac{1}{z} - \gamma + \frac{1}{2}(\gamma^2 + \zeta(2))z + 
    \mathcal{O}(z^2) 
\end{equation}
we have 
\begin{equation}
\begin{split}
& (\kappa z^{-1})^{-t}\Gamma(-t)
\Gamma(1+\mu+t)\frac{G(1+\nu-\mu-\alpha-t)}{G(\nu-\mu-t)}  \\[1ex]
& = (1-t\log{\kappa z^{-1}}+\ldots)(-t^{-1}-\gamma+\cdots)(\Gamma(1+\mu) + \Gamma^\prime(1+\mu)t+\cdots) \\[1ex]
& \qquad \cdot \frac{(G(1+\nu-\mu-\alpha) - G^\prime(1+\nu-\mu-\alpha)t+\cdots)}{G(\nu-\mu)(1-t G^\prime(\nu-\mu)/G(\nu-\mu)+\cdots)} \\[1ex]
& =-\frac{1}{t} \Gamma(1+\mu)\frac{G(1+\nu-\mu-\alpha)}{G(\nu-\mu)} + 
\mathcal{O}(1) ,
\end{split}
\end{equation}
and so 
\begin{equation}
    \rho_1^{\text{(i)}} = -(\kappa z)^{-(1+\mu)} \Gamma(1+\mu)\frac{G(1+\nu-\mu-\alpha)}{G(\nu-\mu)} .
\end{equation}

For eq.\eqref{residue.case.ii}, let us rewrite it as 
\begin{equation}
    \rho_1^{\text{(ii)}} = (\kappa z^{-1})^{-(1+\nu)} 
    \underset{t=0}{\operatorname{Res}} \, (\kappa z^{-1})^{-t} 
    \Gamma(\mu-\nu-t)\Gamma(1+\nu+t) \frac{G(1-\alpha -t)}{G(-t)} . 
\end{equation}
Using eq.\eqref{barnesg.1} we have 
\begin{equation}
\begin{split}
& (\kappa z^{-1})^{-t} 
    \Gamma(\mu-\nu-t)\Gamma(1+\nu+t) \frac{G(1-\alpha -t)}{G(-t)}\\[1ex]
& = (1-t\log\kappa z^{-1} + \cdots) (\Gamma(\mu-\nu)-\Gamma^\prime(\mu-\nu) t + \cdots) 
(\Gamma(1+\nu) + \Gamma^\prime(1+\nu) t + \cdots) \\[1ex]
& \quad \cdot \frac{G(1-\alpha) - G^\prime(1-\alpha)t + \cdots}{(-t) (1-
(\gamma+(1/2)(\log{2\pi}-1) t + \cdots)} \\[1ex]
& -\frac{1}{t}\Gamma(\mu-\nu)\Gamma(1+\nu) G(1-\alpha) 
\end{split}
\end{equation}
which gives 
\begin{eqnarray}
\rho_1^{\text{(ii)}} = -(\kappa z^{-1})^{-(1+\nu)} \Gamma(\mu-\nu)\Gamma(1+\nu) G(1-\alpha) .
\end{eqnarray}

Let us rewrite eq.\eqref{residues.case.iii} as 
\begin{equation}
\rho_1^{\text{(iii)}} = (\kappa z^{-1})^{-(1+\nu)} 
\underset{t=0}{\operatorname{Res}} \, (\kappa z^{-1})^{-t} \Gamma(-t) 
\Gamma(1+\nu+t) \frac{G(1-\alpha-t)}{G(-t)} . 
\end{equation}
Using eq.\eqref{barnesg.1} and eq.\eqref{gamma.series} we have
\begin{equation}
\begin{split}
& (\kappa z^{-1})^{-t} \Gamma(-t) 
\Gamma(1+\nu+t) \frac{G(1-\alpha-t)}{G(-t)} \\[1ex]
& = (1-t\log\kappa z^{-1} + \cdots) (-t^{-1} - \gamma + \cdots) 
(\Gamma(1+\nu) + \Gamma^\prime(1+\nu)t + \cdots) \\[1ex]
& \quad \cdot \frac{G(1-\alpha) - G^\prime(1-\alpha) t + \cdots)}{(-t)(1-(\gamma+(1/2)(\log{2\pi-1))t + \cdots) }} \\[1ex]
& = \frac{\Gamma(1+\nu)G(1-\alpha)}{t^2} + 
\frac{1}{t}\bigg[ \Gamma(1+\nu)G(1-\alpha)\log z/k + \gamma \Gamma(1+\nu) G(1-\alpha) \\[1ex] 
& \quad  + \Gamma^\prime(1+\nu) G(1-\alpha) 
+ \Gamma(1+\nu) G^\prime(1-\alpha) -\left(\gamma+\frac{1}{2}(\log{2\pi}-1)\right) \Gamma(1+\nu)G(1-\alpha) \bigg] .
\end{split}
\end{equation}
Finally using the definition of the digamma function and eq.\eqref{barnesg.3} we obtain 
\begin{equation}
\rho_1^{\text{(iii)}} = 
(\kappa z^{-1})^{-(1+\nu)} G(1-\alpha) \Gamma(1+\nu) 
[2\gamma + \psi(1+\nu) - \alpha + \alpha \psi(1-\alpha) + \log{z/k}] .
\end{equation}


\end{appendices}


\begin{thebibliography}{99}
\bibitem{BAR} E.W. Barnes, Genesis of the double gamma function, \textit{Proc. London Math. Soc.} \textbf{31} 358-381 (1900).



\bibitem{BLV} L. Beghin, N. Leonenko and J. Vaz, Relaxation equations with stretched non-local operators: renewals and time-changed processes, \textit{Journal of Theoretical Probability}, \textbf{39}, 22 (2026).

\bibitem{Simon} L. Boudabsa and T. Simon, Some properties of the Kilbas-Saigo function, In: Special Functions with Applications to Mathematical Physics II, \textit{Mathematics}, \textbf{9}, 1-24 (2021). 


\bibitem{Choi} J. Choi and H.M. Srivastava, Certain classes of
series involving the zeta function. Journal of 
Mathematical Analysis and Applications \textbf{231}, 91-117 
(1999). 


\bibitem{Diethelm} K. Diethelm, \textit{The Analysis of Fractional Differential Equations},
Springer-Verlag (2010). 

\bibitem{Ferreira} C. Ferreira and J. L. L\'opez, An asymptotic expansion of the double gamma
function, \textit{Journal of Approximation Theory}, \textbf{111}, 298-314 (2001).



\bibitem{Gautschi} W. Gautschi, Some elementary inequalities relating to the gamma and incomplete gamma function, \textit{Journal of Mathematics and Physics,} \textbf{38}, 
77-81 (1959).

\bibitem{HSF} J. Letemplier and T. Simon, On the law of homogeneous stable functionals, \textit{ESAIM: Probability and Statistics,} \textbf{23}, 82-111 (2019).

\bibitem{LUC} Y. Luchko and J. J. Trujillo, Caputo-type modification of the Erd\'elyi-Kober fractional derivative, \textit{Fractional Calculus and Applied Analysis}, \textbf{10}, 249-267 (2007).


\bibitem{Kilbas} Kilbas, A.A., Srivastava, H.M., Trujillo, J. J., 
\textit{Theory and Applications of Fractional Differential Equations}, 
Elsevier (2006).

\bibitem{Kochubei2011} A.N. Kochubei, General fractional calculus, evolution equations, and renewal processes. 
Integral Equations and Operator Theory \textbf{71}, 583-600 (2011). 


\bibitem{MAI} F. Mainardi, Applications of integral transforms in fractional diffusion processes, \textit{Integral Transforms and Special Functions},
\textbf{15}, 6, 477-484 (2004).



\bibitem{MeeSch} M.M. Meerschaert and H.P. Scheffler, \textit{Limit Theorems for Sums of Independent Random Vectors},  Wiley, New York, 2001.



\bibitem{MeerschaertSikorskii2012} M.M. Meerschaert and A. Sikorskii, 
\textit{Stochastic Models for Fractional Calculus}, De Gruyter (2012).



\bibitem{MetzlerKlafter2000} R. Metzler and J. Klafter, The random walk's guide 
to anomalous diffusion: a fractional dynamics approach. \textit{Physics Reports}, \textbf{339}, 
1-77 (2000). 


\bibitem{MontrollWeiss1965} E.W. Montroll and G.H. Weiss, Random walks on lattices II, 
\textit{Journal of Mathematical Physics} \textbf{6}, 167-181 (1965).


\bibitem{MUR} A. Mura and G. Pagnini, Characterizations and simulations of a class of stochastic processes to model
anomalous diffusion. \textit{J. Phys. A Math. Theor}, 285003, 41, 22 (2008).


\bibitem{Podlubny} I. Podlubny, \textit{Fractional Differential Equations}, 
Academic Press (1999). 


\bibitem{Sandow} S. Sandow, K. Namtomah and B. Seidu, On some analytical properties of the Barnes $G$ function, 
\textit{Asia Pacific Journal of Mathematics}, \textbf{8}, 7 (2021). 

\bibitem{SCH} W.R. Schneider, \textit{Grey noise}. In: Albeverio, S., Fenstad, J.E., Holden, H., Lindstrom, T. (eds.) Ideas and Methods in Mathematical Analysis, Stochastics and Applications, vol. I, pp. 261-282. Cambridge
University Press, Cambridge (1990).



 
\bibitem{Toaldo2015} B. Toaldo, Convolution-type derivatives, 
hitting-times of subordinators and time-changed $C_0$-semigroups, \textit{Potential Analysis}, \textbf{42}, 115-140 (2015). 

\bibitem{Vaz-Fox-Barnes} J. Vaz, A generalization of 
the Fox H-function. To appear in Theory of Probability and Mathematical Statistics. 
\texttt{https://doi.org/10.48550/arXiv.2510.15920} 


\bibitem{VazCapelas} J. Vaz and E. Capelas de Oliveira, On fractional differential equations, 
dimensional analysis, and the double gamma function, \textit{Nonlinear Dynamics}, \textbf{113}, 
34305-34320 (2025). 



\end{thebibliography}
\end{document}